\documentclass[10pt, reqno]{amsart}

\def\BibTeX{{\rm B\kern-.05em{\sc i\kern-.025em b}\kern-.08em
    T\kern-.1667em\lower.7ex\hbox{E}\kern-.125emX}}

\newtheorem{thm}{Theorem}[section]

\newtheorem{lem}[thm]{Lemma}
\newtheorem{prop}[thm]{Proposition}
\newtheorem{cond}[thm]{Condition}

\theoremstyle{definition}

\theoremstyle{remark}
\newtheorem{rem}{Remark}[section]

\numberwithin{equation}{section}

    \newcommand{\floor}[1]{\lfloor#1\rfloor}

    \newcommand{\EE}{\mathbb{E}}

    \newcommand{\Exp}{\operatorname{E}}
    \newcommand{\E}{\Exp}

    \renewcommand{\Pr}{\operatorname{P}}

    \newcommand{\eqd}{\stackrel{d}{=}}
    \newcommand{\dto}{\xrightarrow{d}}
    \newcommand{\wto}{\xrightarrow{w}}
    \newcommand{\vto}{\xrightarrow{v}}
    \newcommand{\fidi}{\xrightarrow{\text{fidi}}}

    \newcommand{\eind}{\stackrel{d}{=}}
    \newcommand{\rmd}{\mathrm{d}}

\newcommand{\be}{\begin{equation}}
    \newcommand{\ee}{\end{equation}}

\begin{document}

\title[Joint functional convergence for moving averages] 
{Joint functional convergence of partial sums and maxima for moving averages with weakly dependent heavy-tailed innovations and random coefficients}

%
\author{Danijel Krizmani\'{c}}

\address{Danijel Krizmani\'{c}\\ Faculty of Mathematics\\
        University of Rijeka\\
        Radmile Matej\v{c}i\'{c} 2, 51000 Rijeka\\
        Croatia}
\email{dkrizmanic@math.uniri.hr}





\begin{abstract}
For moving average processes with random coefficients and heavy-tailed innovations that are weakly dependent in the sense of strong mixing and local dependence condition $D'$ we study joint functional convergence of partial sums and maxima. Under the assumption that all partial sums of the series of coefficients are a.s.~bounded between zero and the sum of the series we derive a functional limit theorem in the space of $\mathbb{R}^{2}$--valued c\`{a}dl\`{a}g functions
on $[0, 1]$ with the Skorokhod weak $M_{2}$ topology. 
\end{abstract}

\maketitle

\section{Introduction}
\label{intro}

It is known that the joint partial sums and maxima processes constructed from i.i.d.~regularly varying random variables
with the tail index $\alpha \in (0,2)$
converge weakly in the space $D([0, 1], \mathbb{R}^{2})$ of $\mathbb{R}^{2}$--valued c\`adl\`ag functions on $[0, 1]$ with the Skorokhod $J_{1}$ topology, with the limit consisting of a stable L\'{e}vy process and an extremal process, see Chow and Teugels~\cite{ChTe79} and Resnick~\cite{Re86}.

The joint functional convergence holds also in the weakly dependent case. Anderson and Turkman~\cite{AnTu91},~\cite{AnTu95} studied weak convergence of the joint partial sums and maxima processes in the case when the underlying random variables are strongly mixing, and in the heavy-tailed case, under Leadbetter's $D$ and $D'$ dependence conditions familiar from extreme value theory. Conditions $D$ and $D'$ are quite restrictive, since they exclude m-dependent sequences.

Recently, Krizmani\'{c}~\cite{Kr20} showed that for a regularly varying sequence of dependent random variables with index $\alpha \in (0,2)$, for which clusters of high-threshold excesses can be broken down into asymptotically independent blocks, the joint stochastic processes of partial sums and maxima converge in the space $D([0,1], \mathbb{R}^{2})$ endowed with the Skorokhod weak $M_{1}$ topology under the condition that all extremes within each cluster of big values have the same sign. This topology is weaker than the more commonly used Skorokhod $J_{1}$ topology, the latter being appropriate when there is no clustering of extremes. This result extends the functional limit theorem obtained by Krizmani\'{c}~\cite{Kr18-0} in the special case of linear processes with i.i.d.~regularly varying innovations and all (deterministic) coefficients of the same sign. Even when these coefficients are not of the same sign, but all partial sums of the series of coefficients are bounded between zero and the sum of the series, joint functional convergence still holds, but with respect to the weaker Skorokhod $M_{2}$ topology, see Krizmani\'{c}~\cite{Kr18}.

In this paper we study joint functional convergence of partial sums and maxima for linear or moving averages processes with weakly dependent innovations and random coefficients. In proving the joint functional convergence for these processes we will rely on already established marginal functional convergence for partial sums when the innovations are weakly dependent in the sense of strong mixing and local dependence condition $D'$ (see Krizmani\'{c}~\cite{Kr22-1}) and for partial maxima of linear processes with i.i.d.~innovations (see Krizmani\'{c}~\cite{Kr22-2}).

We proceed by stating the problem precisely.
Let $(Z_{i})_{i \in \mathbb{Z}}$ be a strictly stationary sequence of regularly varying random variables with index of regular variation $\alpha \in (0,2)$.
This means that
\begin{equation}\label{e:regvar}
 \Pr(|Z_{i}| > x) = x^{-\alpha} L(x), \qquad x>0,
\end{equation}
where $L$ is a slowly varying function at $\infty$. Let $(a_{n})$ be a sequence of positive real numbers such that
\be\label{eq:niz}
n \Pr (|Z_{1}|>a_{n}) \to 1,
\ee
as $n \to \infty$. Then regular
variation of $Z_{i}$ can be expressed in terms of
vague convergence of measures on $\EE = \overline{\mathbb{R}} \setminus \{0\}$:
\begin{equation}
  \label{eq:onedimregvar}
  n \Pr( a_n^{-1} Z_i \in \cdot \, ) \vto \mu( \, \cdot \,) \qquad \textrm{as} \ n \to \infty,
\end{equation}
where $\mu$ is a measure on $\EE$  given by
\begin{equation}
\label{eq:mu}
  \mu(\rmd x) = \bigl( p \, 1_{(0, \infty)}(x) + r \, 1_{(-\infty, 0)}(x) \bigr) \, \alpha |x|^{-\alpha-1} \, \rmd x,
\end{equation}
with
\be\label{eq:pq}
p =   \lim_{x \to \infty} \frac{\Pr(Z_i > x)}{\Pr(|Z_i| > x)} \qquad \textrm{and} \qquad
  r =   \lim_{x \to \infty} \frac{\Pr(Z_i < -x)}{\Pr(|Z_i| > x)}.
\ee
We study
the moving average process with random coefficients, defined by
\begin{equation}\label{e:MArandom}
X_{i} = \sum_{j=0}^{\infty}C_{j}Z_{i-j}, \qquad i \in \mathbb{Z},
\end{equation}
where
$(C_{i})_{i \geq 0 }$ is a sequence of random variables independent of $(Z_{i})$ such that the above series is a.s.~convergent. One well-known sufficient condition for that is
\begin{equation}\label{e:momcond}
\sum_{j=0}^{\infty} \mathrm{E} |C_{j}|^{\delta} < \infty \qquad \textrm{for some}  \ \delta < \alpha,\,0 < \delta \leq 1.
\end{equation}
The moment condition (\ref{e:momcond}), stationarity of the sequence $(Z_{i})$ and $\mathrm{E}|Z_{1}|^{\beta} < \infty$ for every $\beta \in (0,\alpha)$ (which follows from the regular variation property and Karamata's theorem) imply the a.s.~convergence of the series in (\ref{e:MArandom}), since
$$ \mathrm{E}|X_{i}|^{\delta} \leq \sum_{j=0}^{\infty} \mathrm{E}|C_{j}|^{\delta} \mathrm{E}|Z_{i-j}|^{\delta} = \mathrm{E}|Z_{1}|^{\delta} \sum_{j=0}^{\infty}\mathrm{E}|C_{j}|^{\delta} < \infty.$$
Condition (\ref{e:momcond}) implies also the a.s.~convergence of the series $\sum_{j=0}^{\infty}C_{j}$.
Another condition that assures the a.s.~convergence of the series in the definition of moving average processes with
$$\begin{array}{rl}
 \nonumber \mathrm{E}(Z_{1})=0, & \quad \textrm{if} \ \alpha \in (1,2),\\[0.2em]
 \nonumber Z_{1} \ \textrm{is symmetric}, & \quad \textrm{if} \ \alpha =1,
\end{array}$$
 and a.s.~bounded coefficients
 can be deduced from the results in Astrauskas~\cite{At83}:
$$ \sum_{j =0}^{\infty}c_{j}^{\alpha} L(c_{j}^{-1}) < \infty,$$
where $(c_{j})$ is a sequence of positive real numbers such that $|C_{j}| \leq c_{j}$ a.s.~for all $j$ (c.f.~Balan et al.~\cite{BJL16}).

If the innovations $Z_{i}$'s are independent, Krizmani\'{c}~\cite{Kr19} derived a (marginal) functional limit theorem for the partial sum stochastic process
\be\label{eq:defVn1}
V_{n}(t) = \frac{1}{a_{n}}  \sum_{i=1}^{\floor {nt}}X_{i}, \qquad t \in [0,1],
\ee
in the space $D([0,1], \mathbb{R})$ of real--valued right continuous functions on $[0,1]$ with left limits, endowed with the Skorohod $M_{2}$ topology, under some usual regularity conditions and the assumption that all partial sums of the series $C= \sum_{i=0}^{\infty}C_{i}$ are a.s.~bounded between zero and the sum of the series, i.e.
\be\label{eq:InfiniteMAcond}
0 \le \sum_{i=0}^{s}C_{i} \Bigg/ \sum_{i=0}^{\infty}C_{i} \le 1 \ \ \textrm{a.s.} \qquad \textrm{for every} \ s=0, 1, 2 \ldots.
\ee
More precisely,
\be\label{eq:fconvVn}
 V_{n}(\,\cdot\,) \dto \widetilde{C} V(\,\cdot\,) \qquad \textrm{as} \ n \to \infty,
\ee
in $D([0,1], \mathbb{R})$ endowed with the $M_{2}$ topology,
where $V$ is an $\alpha$--stable L\'{e}vy process with characteristic triple $(0, \mu, b)$, with $\mu$ as in $(\ref{eq:mu})$,
$$ b = \left\{ \begin{array}{cc}
                                   0, & \quad \alpha = 1,\\[0.4em]
                                   (p-r)\frac{\alpha}{1-\alpha}, & \quad \alpha \in (0,1) \cup (1,2),
                                 \end{array}\right.$$
    and $\widetilde{C}$ is a random variable, independent of $V$, such that $\widetilde{C} \eind C$.
 When the sequence of coefficients $(C_{j})$ is deterministic, relation (\ref{eq:fconvVn}) reduces to
 $$ V_{n}(\,\cdot\,) \dto C V(\,\cdot\,) \qquad \textrm{as} \ n \to \infty$$
 (see Proposition 3.2 in Krizmani\'{c}~\cite{Kr19}). Simplifying notation, we sometimes omit brackets and write $V_{n} \dto CV$. This functional convergence, as shown by Avram and Taqqu~\cite{AvTa92}, can not be strengthened to the Skorokhod $J_{1}$ convergence on $D([0,1], \mathbb{R})$, but if all coefficients are nonnegative, then it holds in the $M_{1}$ topology.

Further, (marginal) functional convergence of partial maxima processes in the i.i.d.~case was obtained in Krizmani\'{c}~\cite{Kr22-2}, i.e.~under some standard moment conditions on the sequence of coefficients $(C_{i})$ it holds that
\begin{equation}\label{eq:fconvMn}
 M_{n}(\,\cdot\,) \dto  M(\,\cdot\,) \qquad \textrm{as} \ n \to \infty,
\end{equation}
in $D([0,1], \mathbb{R})$ endowed with the Skorokhod $M_{1}$ topology, where
\be\label{eq:defWn}
M_{n}(t) = \left\{ \begin{array}{cc}
                                   \displaystyle \frac{1}{a_{n}} \bigvee_{i=1}^{\floor {nt}}X_{i}, & \quad  \displaystyle t \in \Big[ \frac{1}{n}, 1 \Big],\\[1.4em]
                                   \displaystyle \frac{X_{1}}{a_{n}}, & \quad \displaystyle  t \in \Big[0,  \frac{1}{n} \Big\rangle,
                                 \end{array}\right.
\ee
is the corresponding partial maximum process and
$M = C^{(1)}W^{(1)} \vee C^{(2)}W^{(2)}$,
 where $W^{(1)}$ is an extremal process with exponent measure $\mu_{+}(\rmd x) = p \alpha x^{-\alpha-1} \rmd x$ for $x>0$, $W^{(2)}$ is an extremal process with exponent measure $\mu_{-}(\rmd x) = r \alpha x^{-\alpha-1} \rmd x$ for $x>0$, and $(C^{(1)}, C^{(2)})$ is a two dimensional random vector, independent of $(W^{(1)}, W^{(2)})$, such that
$(C^{(1)}, C^{(2)}) \eind (C_{+}, C_{-})$, with
\begin{equation}\label{e:Cplusminus}
C_{+}= \max \{ C_{j} \vee 0 : j \geq 0 \} \qquad \textrm{and} \qquad C_{-}= \max \{ -C_{j} \vee 0 : j \geq 0 \}.
\end{equation}

Recently, in Krizmani\'{c}~\cite{Kr22-1} functional convergence in (\ref{eq:fconvVn}) was extended to the case when the innovations $Z_{i}$ are weakly dependent, in the sense that $(Z_{i})$ is a strongly mixing sequence which satisfies the local dependence condition $D'$ as is given in Davis~\cite{Da83}:
 \begin{equation}\label{e:D'cond}
 \lim_{k \to \infty} \limsup_{n \to \infty}~n \sum_{i=1}^{\lfloor n/k \rfloor} \Pr \bigg( \frac{|Z_{0}|}{a_{n}} > x, \frac{|Z_{i}|}{a_{n}} >x \bigg) = 0 \qquad \textrm{for all} \ x >0.
 \end{equation}
 For instance, a process which is an instantaneous function of a stationary Gaussian process with covariance function $r_{n}$ behaving like $r_{n} \log n \to 0$ as $n \to \infty$ satisfies condition (\ref{e:D'cond}) (see Davis~\cite{Da83}). Other examples of time series that satisfy (\ref{e:D'cond}), related to stochastic volatility models and ARMAX processes, can be found in Davis and Mikosch~\cite{DaMi09} and Ferreira and Canto e Castro~\cite{FeCa08}. This condition, together with the strong mixing property, assures that, as in the i.i.d.~case, the extremes of the sequence $(Z_{i})$ are isolated. This corresponds to the situation when the extremal index $\theta$ of the sequence $(Z_{i})$, which can be interpreted as the reciprocal mean cluster size of large exceedances (c.f.~Hsing et al.~\cite{HHL88}), is equal to $1$. Recall here that a strictly stationary sequence of random variables $(\xi_{n})$ has extremal index $\theta$ if for every $\tau >0$ there exists a sequence of real numbers $(u_{n})$ such that
\begin{equation*}\label{eq:eindex}
 \lim_{n \to \infty} n \Pr ( \xi_{1} > u_{n}) \to \tau \qquad \textrm{and} \qquad \lim_{n \to \infty} \Pr \bigg( \max_{1 \leq i \leq n} \xi_{i} \leq u_{n} \bigg) \to e^{-\theta \tau}.
\end{equation*}
It holds that $\theta \in [0,1]$. Recall also that a sequence $(\xi_{n})$ is strongly mixing if $\alpha (n) \to 0$ as $n \to \infty$, where
$$\alpha (n) = \sup \{|\Pr (A \cap B) - \Pr(A) \Pr(B)| : A \in \mathcal{F}_{-\infty}^{0}, B \in \mathcal{F}_{n}^{\infty} \}$$
and $\mathcal{F}_{k}^{l} = \sigma( \{ \xi_{i} : k \leq i \leq l \} )$ for $-\infty \leq k \leq l \leq \infty$. For some related results on limit theory for moving averages with random coefficients we refer to Hult and Samorodnitsky~\cite{HuSa08} and Kulik~\cite{Ku06}.

Our aim in this paper is to join the convergence relations (\ref{eq:fconvVn}) and (\ref{eq:fconvMn}) into a single one, or more precisely, to find sufficient conditions on moving average processes with weakly dependent innovations and random coefficients such that, with respect to some Skorokhod topology on $D([0,1], \mathbb{R}^{2})$
\be\label{eq:fconvVnMn}
 L_{n}(\,\cdot\,) := (V_{n}(\,\cdot\,), M_{n}(\,\cdot\,)) \dto (\widetilde{C} V(\,\cdot\,), M(\,\cdot\,)) \qquad \textrm{as} \ n \to \infty.
\ee
Note that if we prove (\ref{eq:fconvVnMn}) then it will follow directly that the functional convergence of partial maxima processes in (\ref{eq:fconvMn}) holds also in the case when the innovations are weakly dependent.


In proving relation (\ref{eq:fconvVnMn}) we will turn our attention first to finite order moving averages, and then to the general case of infinite order moving average processes. For $Z_{1}$ we assume the already mentioned standard regularity conditions:
  \begin{eqnarray}\label{e:oceknula}
    \mathrm{E} Z_{1}=0, & & \textrm{if} \ \ \alpha \in (1,2),  \\
    Z_{1} \ \textrm{is symmetric}, & & \textrm{if} \ \ \alpha=1.\label{e:sim}
  \end{eqnarray}
In the case $\alpha \in [1,2)$ we will need to assume the following condition to deal with small jumps:
 \begin{equation}\label{e:vsvcond}
 \lim_{u \downarrow 0} \limsup_{n \to \infty}~\Pr \bigg[ \max_{1 \leq k \leq n} \bigg| \sum_{i=1}^{k}  \bigg( \frac{Z_{i}}{a_{n}} 1_{\big\{ \frac{|Z_{i}|}{a_{n}} \leq u\big\}} - \mathrm{E} \bigg( \frac{Z_{i}}{a_{n}} 1_{\big\{ \frac{|Z_{i}|}{a_{n}} \leq u \big\}} \bigg) \bigg) \bigg| > \epsilon \bigg]=0
 \end{equation}
for all $\epsilon >0$. This condition holds if the sequence $(Z_{i})$ is $\rho$-mixing at a certain rate (see Lemma 4.8 in Tyran-Kami\'{n}ska~\cite{Ty10a}). In case $\alpha \in (0,1)$ it is a simple consequence of regular variation and Karamata's theorem. Similar conditions are standardly used in the limit theory for partial sums, see~\cite{AvTa92, BKS, DuRe78, Ty10a}.
For infinite order moving averages, beside condition (\ref{e:momcond}) we will require also some other moment conditions, which will be specified latter in Section~\ref{S:InfiniteMA}.

 Since the stochastic processes $V_{n}$ and $M_{n}$ converge (separately) in the space $D([0,1], \mathbb{R})$ equipped with the $M_{2}$ topology, for the convergence in relation (\ref{eq:fconvVnMn}) we will use the weak $M_{2}$ topology.
  Since for partial maxima processes, functional convergence actually holds in the stronger $M_{1}$ topology, it is possible to obtain also a joint convergence of $L_{n}$ in the $M_{2}$ topology on the first coordinate and in the $M_{1}$ topology on the second coordinate, see Remark~\ref{r:M2M1} below. In general, functional $M_{1}$ convergence of partial sum processes fails to hold, as for instance, in the case of moving averages with i.i.d.~heavy-tailed innovations $Z_{i}$ and deterministic coefficients $C_{0}=1$, $C_{1}=-1$, $C_{2}=1$ and $C_{i}=0$ for $i \geq 3$:
$$ X_{i}= Z_{i} - Z_{i-1} + Z_{i-2}, \qquad i \in \mathbb{Z}.$$
This shows that in general functional convergence of $L_{n}$ does not hold in the (weak) $M_{1}$ topology.

The paper is organized as follows. In Section~\ref{S:Mtop} we recall definitions of weak and strong Skorohod $M_{1}$ and $M_{2}$ topologies. In Section~\ref{S:pomocno} we introduce some basic notions about joint regular variation and point processes, and obtain some auxiliary results that will be used in Section~\ref{S:FiniteMA} in establishing the limiting relation (\ref{eq:fconvVnMn}) for finite order moving average processes with weakly dependent innovations and random coefficients. Then in Section~\ref{S:InfiniteMA} we
extend this result to infinite order moving average processes.

\section{Skorokhod $M_{1}$ and $M_{2}$ topologies}\label{S:Mtop}

We start with the definition of the Skorokhod weak $M_{2}$ topology in a general space $D([0,1], \mathbb{R}^{d})$ of $\mathbb{R}^{d}$--valued c\`{a}dl\`{a}g functions on
$[0,1]$. It is standardly defined using completed graphs and their parametric representations.

For $x \in D([0,1],
\mathbb{R}^{d})$ the completed (thick) graph of $x$ is the set
\[
  G_{x}
  = \{ (t,z) \in [0,1] \times \mathbb{R}^{d} : z \in [[x(t-), x(t)]]\},
\]
where $x(t-)$ is the left limit of $x$ at $t$ and $[[a,b]]$ is the product segment, i.e.
$[[a,b]]=[a_{1},b_{1}] \times [a_{2},b_{2}] \ldots \times [a_{d},b_{d}]$
for $a=(a_{1}, a_{2}, \ldots, a_{d}), b=(b_{1}, b_{2}, \ldots, b_{d}) \in
\mathbb{R}^{d}$. We define an
order on the graph $G_{x}$ by saying that $(t_{1},z_{1}) \le
(t_{2},z_{2})$ if either (i) $t_{1} < t_{2}$ or (ii) $t_{1} = t_{2}$
and $|x_{j}(t_{1}-) - z_{1j}| \le |x_{j}(t_{2}-) - z_{2j}|$
for all $j=1, 2, \ldots, d$. The relation $\le$ induces only a partial
order on the graph $G_{x}$. A weak $M_{2}$ parametric representation
of the graph $G_{x}$ is a continuous function $(r,u)$
mapping $[0,1]$ into $G_{x}$, such that $r$ is nondecreasing with $r(0)=0$, $r(1)=1$ and $u(1)=x(1)$ ($r$ is the
time component and $u$ the spatial component). Denote by $\Pi_{w}^{M_{2}}(x)$ the set of weak $M_{2}$
parametric representations of the graph $G_{x}$. For $x_{1},x_{2}
\in D([0,1], \mathbb{R}^{d})$ define
\[
  d_{w}(x_{1},x_{2})
  = \inf \{ \|r_{1}-r_{2}\|_{[0,1]} \vee \|u_{1}-u_{2}\|_{[0,1]} : (r_{i},u_{i}) \in \Pi_{w}^{M_{2}}(x_{i}), i=1,2 \},
\]
where $\|x\|_{[0,1]} = \sup \{ \|x(t)\| : t \in [0,1] \}$ for $x \colon [0,1] \to \mathbb{R}^{k}$, with $\| \cdot \|$ denoting the max-norm on $\mathbb{R}^{k}$. Now we define the weak $M_{2}$ topology sequentially by saying that a sequence $(x_{n})_{n}$ converges to $x$ in $D([0,1], \mathbb{R}^{d})$ in the weak Skorokhod $M_{2}$ (or shortly $WM_{2}$)
topology if $d_{w}(x_{n},x)\to 0$ as $n \to \infty$.

If we replace above the graph $G_{x}$ with the completed (thin) graph
\[
  \Gamma_{x}
  = \{ (t,z) \in [0,1] \times \mathbb{R}^{d} : z= \lambda x(t-) + (1-\lambda)x(t) \ \text{for some}\ \lambda \in [0,1] \},
\]
and a weak $M_{2}$ parametric representation with a strong $M_{2}$ parametric representation (i.e. a continuous function $(r,u)$ mapping $[0,1]$ onto $\Gamma_{x}$ such that $r$ is nondecreasing), then we obtain the standard (or strong) $M_{2}$ topology. This topology is stronger than the weak $M_{2}$ topology, but they coincide if $d=1$.

Note that in $M_{2}$ parametric representations $(r,u)$ we required that only the time component $r$ is nondecreasing. If we also require that the spatial component $u$ is nondecreasing, then we obtain $M_{1}$ parametric representations and (weak and strong) Skorokhod $M_{1}$ topologies, which are stronger than the corresponding $M_{2}$ topologies.

Often the following characterization of the $M_{2}$ topology with the Hausdorff metric on the spaces of graphs is useful. For $x_{1},x_{2} \in D([0,1], \mathbb{R}^{d})$, the $M_{2}$ distance between $x_{1}$ and $x_{2}$ is given by
$$ d_{M_{2}}(x_{1}, x_{2}) = \bigg(\sup_{a \in \Gamma_{x_{1}}} \inf_{b \in \Gamma_{x_{2}}} d(a,b) \bigg) \vee \bigg(\sup_{a \in \Gamma_{x_{2}}} \inf_{b \in \Gamma_{x_{1}}} d(a,b) \bigg),$$
where $d$ is the metric induced by the maximum norm on $\mathbb{R}^{d+1}$.
The metric $d_{M_{2}}$ induces the strong $M_{2}$ topology.
The weak $M_{2}$ topology on $D([0,1], \mathbb{R}^{d})$
coincides with the (product) topology induced by the metric
\begin{equation}\label{e:defdp}
 d_{p}^{M_{2}}(x_{1},x_{2})= \max_{j=1,\ldots,d}d_{M_{2}}(x_{1j},x_{2j})
\end{equation}
 for $x_{i}=(x_{i1}, \ldots, x_{id}) \in D([0,1],
 \mathbb{R}^{d})$, $i=1,2$. For detailed discussion of the strong and weak $M_{2}$ topologies we refer to Whitt~\cite{Whitt02}, sections 12.10--12.11.
 Denote by $D_{\uparrow}([0,1], \mathbb{R}^{d})$ the subspace of functions $x$ in $D([0,1], \mathbb{R}^{d})$ for which the coordinate functions $x_{i}$ are non-decreasing for all $i=1,\ldots,d$. For simplicity of notation let $D^{d} \equiv D([0,1], \mathbb{R}^{d})$ and $D^{d}_{\uparrow} \equiv D_{\uparrow}([0,1], \mathbb{R}^{d})$.

Similar to relation (\ref{e:defdp}) for the weak $M_{2}$ topology, the weak $M_{1}$ topology on $D^{d}$ coincides with the topology induced by the metric
\begin{equation}\label{e:defdpM1}
 d_{p}^{M_{1}}(x_{1},x_{2})= \max_{j=1,\ldots,d}d_{M_{1}}(x_{1j},x_{2j})
\end{equation}
 for $x_{i}=(x_{i1}, \ldots, x_{id}) \in D^{d}$, $i=1,2$ (see Whitt~\cite{Whitt02}, Theorem 12.5.2). Here $d_{M_{1}}$ denotes the $M_{1}$ metric on $D^{1}$, defined by
$$ d_{M_{1}}(y_{1},y_{2})
  = \inf \{ \|r_{1}-r_{2}\|_{[0,1]} \vee \|u_{1}-u_{2}\|_{[0,1]} : (r_{i},u_{i}) \in \Pi^{M_{1}}(y_{i}), i=1,2 \}$$
for $y_{1},y_{2} \in D^{1}$, where $\Pi^{M_{1}}(y)$ is the set of $M_{1}$ parametric representations of the completed graph $\Gamma_{y}$, i.e.~continuous nondecreasing functions $(r,u)$ mapping $[0,1]$ onto $\Gamma_{y}$.

In the next section we will use the following three lemmas. The first one is about preservation of weak $M_{1}$ convergence of stochastic processes under transformations that add certain c\`adl\`ag functions to the first components of the underlying processes. This results is a simple consequence of $M_{1}$ continuity of addition, the continuous mapping theorem and Slutsky's theorem. It can be proven similarly as Lemma 1 in Krizmani\'{c}~\cite{Kr18} (for the $M_{2}$ convergence). The remaining two lemmas deal with $M_{1}$ continuity of multiplication and maximum of two c\`{a}dl\`{a}g functions. The first one is based on Theorem 13.3.2 in Whitt~\cite{Whitt02}, and the second one follows easily from the fact that for monotone functions $M_{1}$ convergence is equivalent to point-wise convergence in a dense subset of $[0,1]$ including $0$ and $1$ (cf.~Whitt~\cite{Whitt02}, Corollary 12.5.1). Denote by $\textrm{Disc}(x)$ the set of discontinuity points of $x \in D^{1}$.

\begin{lem}\label{l:weakM1transf}
Let $(A_{n}, B_{n}, C_{n})$, $n=0,1,2,\ldots$, be stochastic processes in $D^{3}$ such that, as $n \to \infty$,
\begin{equation}\label{e:lemM2}
(A_{n}, B_{n}, C_{n}) \dto (A_{0}, B_{0}, C_{0})
\end{equation}
in $D^{3}$ with the weak $M_{1}$ topology. Suppose $x_{n}$, $n=0,1,2,\ldots$, are elements of $D^{1}$ with $x_{0}$ being continuous, such that, as $n \to \infty$,
$$x_{n}(t) \to x_{0}(t) $$
uniformly in $t$. Then
$$ (A_{n} +x_{n}, B_{n}, C_{n}) \dto (A_{0} + x_{0}, B_{0}, C_{0})$$
in $D^{3}$ with the weak $M_{1}$ topology.
\end{lem}

\begin{lem}\label{l:contmultpl}
Suppose that $x_{n} \to x$ and $y_{n} \to y$ in $D^{1}$ with the $M_{1}$ topology. If for each $t \in \textrm{Disc}(x) \cap \textrm{Disc}(y)$, $x(t)$, $x(t-)$, $y(t)$ and $y(t-)$ are all nonnegative and $[x(t)-x(t-)][y(t)-y(t-)] \geq 0$, then $x_{n}y_{n} \to xy$ in $D^{1}$ with the $M_{1}$ topology, where $(xy)(t) = x(t)y(t)$ for $t \in [0,1]$.
\end{lem}

\begin{lem}\label{l:weakM2transf}
The function $h \colon D^{2}_{\uparrow} \to D^{1}_{\uparrow}$ defined by
$h(x,y)= x \vee y$, where
$$ (x \vee y)(t) = x(t) \vee y(t), \qquad t \in [0,1],$$
is continuous
when $D^{2}_{\uparrow}$ is endowed with the weak $M_{1}$ topology and $D^{1}_{\uparrow}$ is endowed with the standard $M_{1}$ topology.
\end{lem}

\section{Joint regular variation, point processes and sum-maximum functional}
\label{S:pomocno}

We say that a strictly stationary sequence of random variables $(\xi_{n})_{n \in \mathbb{Z}}$ is (jointly) regularly varying with index
$\alpha >0$ if for any nonnegative integer $k$ the
$k$-dimensional random vector $\xi = (\xi_{1}, \ldots , \xi_{k})$ is
multivariate regularly varying with index $\alpha$, i.e.\ there
exists a random vector $\Theta$ on the unit sphere
$\mathbb{S}^{k-1} = \{ x \in \mathbb{R}^{k} : \|x\|=1 \}$ such
that for every $u >0$, as $x \to \infty$,
 \begin{equation}\label{e:regvar1}
   \frac{\Pr(\|\xi\| > ux,\,\xi / \| \xi \| \in \cdot \, )}{\Pr(\| \xi \| >x)}
    \wto u^{-\alpha} \Pr( \Theta \in \cdot \,),
 \end{equation}
where the arrow ''$\wto$'' denotes the weak convergence of finite measures.
There is a convenient characterization of joint regular variation due to Basrak and Segers~\cite{BaSe}:~it is necessary and
sufficient that there exists a process $(Y_n)_{n \in \mathbb{Z}}$
with $\Pr(|Y_0| > y) = y^{-\alpha}$ for $y \geq 1$ such that, as $x
\to \infty$,
\begin{equation}\label{e:tailprocess1}
  \bigl( (x^{-1}\ \xi_n)_{n \in \mathbb{Z}} \, \big| \, | \xi_0| > x \bigr)
  \fidi (Y_n)_{n \in \mathbb{Z}},
\end{equation}
where "$\fidi$" denotes convergence of finite-dimensional
distributions. The process $(Y_{n})$ is called
the tail process of $(\xi_{n})$.

Let $(Z_{i})_{i \in \mathbb{Z}}$ be a strictly stationary and strongly mixing sequence of regularly varying random variables with index $\alpha \in (0,2)$, such that the local dependence condition $D'$ and conditions $(\ref{e:oceknula})$ and $(\ref{e:sim})$ hold. If $\alpha \in [1,2)$, also suppose that condition $(\ref{e:vsvcond})$ holds.
Condition $D'$ and strong mixing imply that the extremes of the sequence $(Z_{i})$ are isolated, i.e.~$\theta=1$ (see Leadbetter and Rootz\'{e}n~\cite{LeRo88}, page 439, and Leadbetter et al.~\cite{LLR83}, Theorem 3.4.1). They also imply that $(Z_{i})$ is jointly regularly varying with the tail process $(Y_{i})$ being the same as in the i.i.d.~case, that is, $Y_{i}=0$ for $i \neq 0$, and $Y_{0}$  as described above (Basrak et al.~\cite{BKS}, Example 4.1). This in particular means that $(Y_{i})$ has no two values of the opposite sign.

Define the time-space point processes
\begin{equation*}\label{E:ppspacetime}
 N_{n} = \sum_{i=1}^{n} \delta_{(i / n,\,Z_{i} / a_{n})} \qquad \textrm{for all} \ n\in \mathbb{N},
\end{equation*}
with $a_{n}$ as in (\ref{eq:niz}). The point process convergence for the sequence $(N_{n})$ on the space $[0,1] \times \EE$ was obtained by Basrak and Tafro~\cite{BaTa16} under joint regular variation and the following two weak dependence conditions.

\begin{cond}\label{c:mixcond1}
There exists a sequence of positive integers $(r_{n})$ such that $r_{n} \to \infty $ and $r_{n} / n \to 0$ as $n \to \infty$ and such that for every nonnegative continuous function $f$ on $[0,1] \times \mathbb{E}$ with compact support, denoting $k_{n} = \lfloor n / r_{n} \rfloor$, as $n \to \infty$,
\begin{equation}\label{e:mixcon}
 \E \biggl[ \exp \biggl\{ - \sum_{i=1}^{n} f \biggl(\frac{i}{n}, \frac{Z_{i}}{a_{n}}
 \biggr) \biggr\} \biggr]
 - \prod_{k=1}^{k_{n}} \E \biggl[ \exp \biggl\{ - \sum_{i=1}^{r_{n}} f \biggl(\frac{kr_{n}}{n}, \frac{Z_{i}}{a_{n}} \biggr) \biggr\} \biggr] \to 0.
\end{equation}
\end{cond}

\begin{cond}\label{c:mixcond2}
There exists a sequence of positive integers $(r_{n})$ such that $r_{n} \to \infty $ and $r_{n} / n \to 0$ as $n \to \infty$ and such that for every $u > 0$,
\begin{equation}
\label{e:anticluster}
  \lim_{m \to \infty} \limsup_{n \to \infty}
  \Pr \biggl( \max_{m \leq |i| \leq r_{n}} | Z_{i} | > ua_{n}\,\bigg|\,| Z_{0}|>ua_{n} \biggr) = 0.
\end{equation}
\end{cond}
Condition~\ref{c:mixcond1} is implied by the strong mixing property (see Krizmani\'{c}~\cite{Kr10},~\cite{Kr16}). Condition~\ref{c:mixcond2} follows from condition $D'$, for the latter implies
\[
    \lim_{n \to \infty} n \sum_{i=1}^{r_{n}} \Pr \bigg( \frac{|Z_{0}|}{a_{n}} > u,
    \frac{|Z_{i}|}{a_{n}} > u \bigg) = 0 \quad \textrm{for all} \ u
    >0,
\]
for any sequence of positive integers $(r_{n})$ such that $r_{n} \to \infty$ and $r_{n} / n \to 0$ as $n \to \infty$.
Therefore, in our case, by Theorem 3.1 in Basrak and Tafro~\cite{BaTa16}, as $n \to \infty$,
\begin{equation}\label{e:BaTa}
N_{n} \dto N = \sum_{i}\sum_{j}\delta_{(T_{i}, P_{i}\eta_{ij})}
\end{equation}
in $[0,1] \times \EE$, where $\sum_{i=1}^{\infty}\delta_{(T_{i}, P_{i})}$ is a Poisson process on $[0,1] \times (0,\infty)$
with intensity measure $Leb \times \nu$ where $\nu(\rmd x) = \alpha
x^{-\alpha-1}1_{(0,\infty)}(x)\,\rmd x$, and $(\sum_{j= 1}^{\infty}\delta_{\eta_{ij}})_{i}$ is an i.i.d.~sequence of point processes in $\EE$ independent of $\sum_{i}\delta_{(T_{i}, P_{i})}$ and with common distribution equal to the distribution of $\sum_{j}\delta_{\widetilde{Y}_{j}/L(\widetilde{Y})}$, where $L(\widetilde{Y})= \sup_{j \in \mathbb{Z}}|\widetilde{Y}_{j}|$ and $\sum_{j}\delta_{\widetilde{Y}_{j}}$ is distributed as $( \sum_{j \in \mathbb{Z}} \delta_{Y_j} \,|\, \sup_{i \le -1} | Y_i| \le 1).$ Taking into account the form of the tail process $(Y_{i})$ it holds that $N=\sum_{i}\delta_{(T_{i}, P_{i}\eta_{i0})}$ with $|\eta_{i0}|=1$. Further, by (\ref{eq:pq}) and (\ref{e:tailprocess1}) we obtain
\begin{eqnarray*}
\Pr(\eta_{i0}=1) &= &\Pr(Y_{0}>0)=\Pr(Y_{0}>1) = \lim_{x \to \infty}  \Pr \bigl(x^{-1}Z_{0}>1 \, \big| \, | Z_{0}| > x \bigr)\\[0.5em]
  &=& \lim_{x \to \infty} \frac{\Pr(Z_0 > x)}{\Pr(|Z_0| > x)}=p,
\end{eqnarray*}
and similarly $\Pr(\eta_{i0}=-1)=r$. Hence, denoting $Q_{i}=\eta_{i0}$, the limiting point process in relation (\ref{e:BaTa}) reduces to
\begin{equation}\label{e:BaTa1}
 N = \sum_{i}\delta_{(T_{i}, P_{i}Q_{i})},
\end{equation}
with $\Pr(Q_{i}=1)=p$ and $\Pr(Q_{i}=-1)=r$. Since the sequence $(Q_{i})$ is independent of the Poisson process $\sum_{i=1}^{\infty}\delta_{(T_{i}, P_{i})}$, an application of Proposition 5.2 and Proposition 5.3 in Resnick~\cite{Resnick07} yields that $N$ is a Poisson process with intensity measure $Leb \times \nu'$ where
\begin{eqnarray*}
  \nu'(\rmd x) &=& \{ \mathrm{E}[Q_{1}^{\alpha} 1_{\{ Q_{1}>0 \}}] 1_{(0,\infty)}(x) + \mathrm{E}[(-Q_{1})^{\alpha} 1_{\{ Q_{1}<0 \}}] 1_{(-\infty,0)}(x) \} \alpha
|x|^{-\alpha-1}\,\rmd x \\[0.6em]
   &=&  (p 1_{(0,\infty)}(x)+r 1_{(-\infty,0)}(x)) \alpha
|x|^{-\alpha-1}\,\rmd x\\[0.6em]
&=& \mu(\rmd x).
\end{eqnarray*}

Fix $0 < u < \infty$ and define the sum-maximum functional
$$ \Phi^{(u)} \colon \mathbf{M}_{p}([0,1] \times \EE) \to D^{1} \times D^{2}_{\uparrow}$$
by
$$ \Phi^{(u)} \Big( \sum_{i}\delta_{(t_{i}, x_{i})} \Big) (t)
  =  \Big( \sum_{t_{i} \leq t}x_{i}\,1_{\{u < |x_{i}| < \infty \}},  \bigvee_{t_{i} \leq t} |x_{i}| 1_{\{ x_{i}>0 \}}, \bigvee_{t_{i} \leq t} |x_{i}| 1_{\{ x_{i}<0 \}} \Big), \quad t \in [0,1]$$
  (with the convention $\vee \emptyset = 0$),
where the space $\mathbf{M}_p([0,1] \times \EE)$ of Radon point
measures on $[0,1] \times \EE$ is equipped with the vague
topology (see Resnick~\cite{Re87}, Chapter 3). Let $\Lambda = \Lambda_{1} \cap \Lambda_{2}$, where
\begin{multline*}
 \Lambda_{1} =
 \{ \eta \in \mathbf{M}_{p}([0,1] \times \EE) :
   \eta ( \{0,1 \} \times \EE) = 0 = \eta ([0,1] \times \{ \pm \infty, \pm u \}) \}, \\[1em]
 \shoveleft \Lambda_{2} =
 \{ \eta \in \mathbf{M}_{p}([0,1] \times \EE) :
  \eta ( \{ t \} \times (u, \infty]) \cdot \eta ( \{ t \} \times [-\infty,-u)) = 0 \\
  \text{for all $t \in [0,1]$} \}.
\end{multline*}
 Observe that the elements
of $\Lambda_2$ have the property that atoms in $[0,1] \times \mathbb{E}_{u}$ with the same time
coordinate are all on the same side of the time axis, where $\EE_{u} = \{ x
\in \EE : |x| >u \}$.

\begin{lem}\label{l:contfunct}
The maximum functional $\Phi^{(u)} \colon \mathbf{M}_{p}([0,1] \times \EE) \to D^{1} \times D^{2}_{\uparrow}$ is continuous on the set
$\Lambda$
when $D^{1} \times D^{2}_{\uparrow}$ is endowed with the weak $M_{1}$ topology.
\end{lem}
\begin{proof}
Take an arbitrary $\eta \in \Lambda$ and suppose that $\eta_{n} \vto \eta$ as $n \to \infty$ in $\mathbf{M}_p([0,1] \times
\EE)$. We need to show that
$\Phi^{(u)}(\eta_n) \to \Phi^{(u)}(\eta)$ in $D^{1} \times D^{2}_{\uparrow}$ according to the weak $M_1$ topology. By
Theorem 12.5.2 in Whitt~\cite{Whitt02}, it suffices to prove that,
as $n \to \infty$,
$$ d_{p}^{M_{1}}(\Phi^{(u)}(\eta_{n}), \Phi^{(u)}(\eta)) =
\max_{k=1, 2, 3}d_{M_{1}}(\Phi^{(u)}_{k}(\eta_{n}),
\Phi^{(u)}_{k}(\eta)) \to 0.$$
Following, with small modifications, the lines in the proof of Lemma~3.2 in Basrak et al.~\cite{BKS} we obtain
$d_{M_{1}}(\Phi^{(u)}_{1}(\eta_{n}), \Phi^{(u)}_{1}(\eta)) \to 0$ as
$n \to \infty$. Since Proposition 3.1 in Krizmani\'{c}~\cite{Kr22-2} implies
$$\max_{k=2,3}d_{M_{1}}(\Phi^{(u)}_{k}(\eta_{n}),
\Phi^{(u)}_{k}(\eta)) \to 0 \qquad \textrm{as} \ n \to \infty,$$
we conclude that $\Phi^{(u)}$ is continuous at $\eta$.
\end{proof}

In the next result we show that a particular stochastic process $\widetilde{L}_{n}$, constructed carefully from the sequence $(Z_{i})$, converges to the limiting process in relation (\ref{eq:fconvVnMn}) in $D^{1} \times D^{1}_{\uparrow}$ with the weak $M_{1}$ topology. Later, as our main result, we will show that the $M_{2}$ distance between processes $L_{n}$ and $\widetilde{L}_{n}$ is asymptotically negligible (as $n$ tends to infinity), which will imply the desired convergence in (\ref{eq:fconvVnMn}).

\begin{prop}\label{p:FLT}
 Let $(Z_{i})_{i \in \mathbb{Z}}$ be a strictly stationary and strongly mixing sequence of regularly varying random variables satisfying $(\ref{e:regvar})$ and $(\ref{eq:pq})$ with $\alpha \in (0,2)$, such that the local dependence condition $D'$ and conditions $(\ref{e:oceknula})$ and $(\ref{e:sim})$ hold. If $\alpha \in [1,2)$, also suppose that condition $(\ref{e:vsvcond})$ holds. Let
 $(C_{i})_{i \geq 0 }$ be a sequence of random variables independent of $(Z_{i})$ such that the series defying the linear process
 $$ X_{i} = \sum_{j=0}^{\infty}C_{j}Z_{i-j}, \qquad i \in \mathbb{Z},$$
  is a.s.~convergent.
Let $$ \widetilde{V}_{n}(t) :=  \sum_{i=1}^{\floor {nt}}\frac{CZ_{i}}{a_{n}} \quad \textrm{and} \quad \widetilde{M}_{n}(t):= \bigvee_{i=1}^{\floor {nt}}\frac{ |Z_{i}|}{a_{n}}(C_{+}1_{\{ Z_{i} >0 \}} + C_{-} 1_{\{ Z_{i}<0 \}}), \qquad t \in [0,1],$$
with $C=\sum_{i=0}^{\infty}C_{i}$, and $C_{+}$ and $C_{-}$ defined in $(\ref{e:Cplusminus})$.
Then, as $n \to \infty$,
\begin{equation*}\label{e:pomkonv}
\widetilde{L}_{n}(\,\cdot\,) := (\widetilde{V}_{n}(\,\cdot\,), \widetilde{M}_{n}(\,\cdot\,)) \dto  (C^{(0)}V(\,\cdot\,), C^{(1)}M^{(1)}(\,\cdot\,) \vee C^{(2)}M^{(2)}(\,\cdot\,))
\end{equation*}
 in $D^{1} \times D^{1}_{\uparrow}$ with the weak $M_{1}$ topology,
where $V$ is an $\alpha$--stable L\'{e}vy process with characteristic triple $(0, \mu, b)$, with $\mu$ as in $(\ref{eq:mu})$,
$$ b = \left\{ \begin{array}{cc}
                                   0, & \quad \alpha = 1,\\[0.4em]
                                   (p-r)\frac{\alpha}{1-\alpha}, & \quad \alpha \in (0,1) \cup (1,2),
                                 \end{array}\right.$$
$M^{(1)}$ and $M^{(2)}$ are extremal processes with exponent measures
$p \alpha x^{-\alpha-1}1_{(0,\infty)}(x)\,dx$ and $r \alpha x^{-\alpha-1}1_{(0,\infty)}(x)\,dx$ respectively, with $p$ and $r$ defined in $(\ref{eq:pq})$, and $(C^{(0)}, C^{(1)}, C^{(2)})$ is a random vector, independent of $(V, M^{(1)}, M^{(2)})$, such that
$(C^{(0)}, C^{(1)}, C^{(2)}) \eind (C, C_{+}, C_{-})$.
\end{prop}

Before the proof of the proposition recall here some basic facts on L\'{e}vy processes and extremal processes. The distribution of a L\'{e}vy process $V$ is characterized by its
characteristic triple (i.e.~the characteristic triple of the infinitely divisible distribution
of $V(1)$). The characteristic function of $V(1)$ and the characteristic triple
$(a, \rho, c)$ are related in the following way:
 $$
  \mathrm{E} [e^{izV(1)}] = \exp \biggl( -\frac{1}{2}az^{2} + icz + \int_{\mathbb{R}} \bigl( e^{izx}-1-izx 1_{[-1,1]}(x) \bigr)\,\rho(\rmd x) \biggr)$$
for $z \in \mathbb{R}$, where $a \ge 0$, $c \in \mathbb{R}$ are constants, and $\rho$ is a measure on $\mathbb{R}$ satisfying
$$ \rho ( \{0\})=0 \qquad \text{and} \qquad \int_{\mathbb{R}}(|x|^{2} \wedge 1)\,\rho(\rmd x) < \infty.$$
 We refer to~\cite{Sa99} for a textbook treatment of
L\'{e}vy processes.
The distribution of an nonnegative extremal process $W$ is characterized by its exponent measure $\rho$ in the following way:
$$ \Pr (W(t) \leq x ) = e^{-t \rho(x,\infty)}$$
for $t>0$ and $x>0$, where $\rho$ is a measure on $(0,\infty)$ satisfying
$ \rho (\delta, \infty) < \infty$
for any $\delta >0$ (see Resnick~\cite{Resnick07}, page 161). If $\rho$ is a null measure, we suppose $W$ is a zero process.

\begin{proof}[Proof of Proposition~\ref{p:FLT}]
Take an arbitrary $0<u<1$, and consider
$$ \Phi^{(u)}(N_{n})(\,\cdot\,) = \bigg( \sum_{i/n \leq\,\cdot}\frac{Z_{i}}{a_{n}} 1_{ \big\{ \frac{|Z_{i}|}{a_{n}} > u \big\} }, \bigvee_{i/n \leq\,\cdot}\frac{Z_{i}}{a_{n}} 1_{ \{Z_{i} > 0 \} }, \bigvee_{i/n \leq\,\cdot}\frac{Z_{i}}{a_{n}} 1_{ \{Z_{i} < 0 \} } \bigg).$$
Since the limiting point process $N$ is a Poisson process, it is almost surely contained in the set $\Lambda$ (see Resnick~\cite{Resnick07}, page 221). Therefore, since by Lemma~\ref{l:contfunct} the sum-maximum functional $\Phi^{(u)}$ is continuous on the set $\Lambda$, the continuous mapping theorem applied to the convergence in (\ref{e:BaTa}) yields $\Phi^{(u)}(N_{n}) \dto \Phi^{(u)}(N)$ in $D^{1} \times D^{2}_{\uparrow}$ under the weak $M_{1}$ topology, i.e.
\begin{eqnarray}\label{e:mainconv0}
 \nonumber \bigg( \sum_{i = 1}^{\lfloor n \, \cdot \, \rfloor} \frac{Z_{i}}{a_{n}}
    1_{ \bigl\{ \frac{|Z_{i}|}{a_{n}} > u \bigr\} }, \bigvee_{i=1}^{\floor{n\,\cdot}} \frac{|Z_{i}|}{a_{n}} 1_{ \{ Z_{i} > 0 \} }, \bigvee_{i=1}^{\floor{n\,\cdot}}\frac{|Z_{i}|}{a_{n}} 1_{\{ Z_{i}<0 \}}  \bigg) &&\\[0.8em]
    & \hspace*{-48em} \dto & \hspace*{-24em} \bigg( \sum_{T_{i} \le \, \cdot} P_{i}Q_{i} 1_{\{ |P_{i}Q_{i}| > u \}}, \bigvee_{T_{i} \leq\,\cdot} |P_{i}Q_{i}| 1_{\{ P_{i}Q_{i} >0 \}}, \bigvee_{T_{i} \leq\,\cdot} |P_{i}Q_{i}| 1_{\{ P_{i}Q_{i}<0 \}} \bigg).
 \end{eqnarray}
 Since $P_{i}>0$ and $|Q_{i}|=1$ for all $i$, the limiting process in (\ref{e:mainconv0}) reduces to
 $$  \bigg( \sum_{T_{i} \le \, \cdot} P_{i}Q_{i} 1_{\{ P_{i} > u \}}, \bigvee_{T_{i} \leq\,\cdot} P_{i} 1_{\{ Q_{i} >0 \}}, \bigvee_{T_{i} \leq\,\cdot} P_{i} 1_{\{ Q_{i}<0 \}} \bigg).$$
Relation (\ref{eq:onedimregvar}) implies, as $n \to \infty$,
\begin{eqnarray}\label{e:conv12}
 \nonumber \floor{nt} \mathrm{E} \Big( \frac{Z_{1}}{a_{n}} 1_{ \big\{ u < \frac{|Z_{1}|}{a_{n}} \leq 1 \big\} } \Big) &=& \frac{\floor{nt}}{n} \int_{u < |x| \leq 1} x\,n\Pr \Big( \frac{Z_{1}}{a_{n}} \in \rmd x \Big) \\[0.6em]
   & \to &  t \int_{u < |x| \leq 1} x\mu(\rmd x)
\end{eqnarray}
for every $t \in [0,1]$, and this convergence is uniform in $t$.
Therefore an application of Lemma~\ref{l:weakM1transf} to (\ref{e:mainconv0}) and (\ref{e:conv12}) yields, as $n \to \infty$,
\begin{multline}\label{e:mainconv}
      L_{n}^{(u)}(\,\cdot\,) := \bigg( \sum_{i = 1}^{\lfloor n \, \cdot \, \rfloor} \frac{Z_{i}}{a_{n}}
    1_{ \bigl\{ \frac{|Z_{i}|}{a_{n}} > u \bigr\} } - \floor{n\,\cdot\,}b_{n}^{(u)}, \bigvee_{i=1}^{\floor{n\,\cdot}} \frac{|Z_{i}|}{a_{n}} 1_{ \{ Z_{i} > 0 \} }, \bigvee_{i=1}^{\floor{n\,\cdot}}\frac{|Z_{i}|}{a_{n}} 1_{\{ Z_{i}<0 \}}  \bigg) \\
    \dto L^{(u)}(\,\cdot\,) :=  \bigg( \sum_{T_{i} \le \, \cdot} P_{i}Q_{i} 1_{\{ P_{i} > u \}} - (\,\cdot\,) b^{(u)}, \bigvee_{T_{i} \leq\,\cdot} P_{i} 1_{\{ Q_{i} >0 \}}, \bigvee_{T_{i} \leq\,\cdot} P_{i} 1_{\{ Q_{i}<0 \}} \bigg)
\end{multline}
in $D^{1} \times D^{2}_{\uparrow}$ under the weak $M_{1}$ topology, where
$$ b_{n}^{(u)} = \mathrm{E} \Big( \frac{Z_{1}}{a_{n}} 1_{ \big\{ u < \frac{|Z_{1}|}{a_{n}} \leq 1 \big\} } \Big) \qquad \textrm{and} \qquad  b^{(u)} = \int_{u < |x| \leq 1} x\mu(\rmd x).$$

Since $N=\sum_{i} \delta_{(T_{i}, P_{i}Q_{i} })$
 is a Poisson process with intensity measure
$ Leb \times \mu$, by the It\^{o} representation of a L\'{e}vy process (see Resnick~\cite{Resnick07}, pages 150--157; and Sato~\cite{Sa99}, Theorem 14.3 and Theorem 19.2), there exists an $\alpha$--stable L\'{e}vy process $V^{(0)}(\,\cdot\,)$ with characteristic triple $(0, \mu, 0)$ such that
$$ \sup_{t \in [0,1]} |L^{(u)1}(t)-V^{(0)}(t)| \to 0$$
almost surely as $u \to 0$, where $L^{(u)1}$ is the first component of the process $L^{(u)}$. Since uniform convergence implies Skorohod $M_{1}$ convergence, we get
\begin{equation}\label{e:dm1}
d_{M_{1}} ( L^{(u)1}, V ) \to 0
\end{equation}
almost surely as $u \to 0$.
Let
$$L^{(0)}(\,\cdot\,) := \bigg( V^{(0)}(\,\cdot\,), \bigvee_{T_{i} \leq\,\cdot} P_{i} 1_{\{ Q_{i} >0 \}}, \bigvee_{T_{i} \leq\,\cdot} P_{i} 1_{\{ Q_{i}<0 \}} \bigg).$$
Then from (\ref{e:dm1}) we obtain
\begin{equation*}
d_{p}^{M_{1}}(L^{(u)}, L^{(0)}) \to 0
\end{equation*}
almost surely as $u \to 0$. Since almost sure convergence implies convergence in distribution, we have, as $u \to 0$,
\begin{equation}\label{e:mainconv2}
 L^{(u)}(\,\cdot\,) \dto L^{(0)}(\,\cdot\,)
\end{equation}
in $D^{1} \times D^{2}_{\uparrow}$ endowed with the weak $M_{1}$ topology.
By Proposition 5.2 and Proposition 5.3 in Resnick~\cite{Resnick07}, the process
$$\sum_{i} \delta_{(T_{i}, P_{i} 1_{\{ Q_{1}>0 \}} ) }$$
 is a Poisson process with intensity measure
$ Leb \times \nu_{1}$, where
$$ \nu_{1}(\rmd x) = \mathrm{E}[(1_{\{ Q_{1}>0 \}})^{\alpha}] \alpha x^{-\alpha-1} 1_{(0,\infty)}(x)\, \rmd x = p \alpha x^{-\alpha-1} 1_{(0,\infty)}(x)\, \rmd x.$$
Therefore the process
$$  M^{(1)}(\,\cdot\,) := \bigvee_{T_{i} \leq\,\cdot} P_{i} 1_{\{ Q_{i}>0 \}}$$
is an extremal process with exponent measure $\nu_{1}$ (see Resnick~\cite{Resnick07}, page 161).
Similarly, the process
$$\sum_{i} \delta_{(T_{i}, P_{i} 1_{\{ Q_{1}<0 \}} )}$$
 is a Poisson process with intensity measure
$ Leb \times \nu_{2}$, where
$$ \nu_{2}(\rmd x) = \mathrm{E}[(1_{\{ Q_{1}<0 \}})^{\alpha}] \alpha x^{-\alpha-1} 1_{(0,\infty)}(x)\, \rmd x = r \alpha x^{-\alpha-1} 1_{(0,\infty)}(x)\, \rmd x,$$
and thus the process
$$ M^{(2)}(\,\cdot\,) := \bigvee_{T_{i} \leq\,\cdot} P_{i} 1_{\{ Q_{i}<0 \}}$$
is an extremal process with exponent measure $\nu_{2}$.

Let
$$   L_{n}^{(0)}(\,\cdot\,) := \bigg( \sum_{i = 1}^{\lfloor n \, \cdot \, \rfloor} \frac{Z_{i}}{a_{n}}
     - \floor{n\,\cdot\,} \mathrm{E} \Big( \frac{Z_{1}}{a_{n}} 1_{ \big\{ \frac{|Z_{1}|}{a_{n}} \leq 1 \big\} } \Big) , \bigvee_{i=1}^{\floor{n\,\cdot}} \frac{|Z_{i}|}{a_{n}} 1_{ \{ Z_{i} > 0 \} }, \bigvee_{i=1}^{\floor{n\,\cdot}}\frac{|Z_{i}|}{a_{n}} 1_{\{ Z_{i}<0 \}}  \bigg).
    $$
If we show that
$$ \lim_{u \to 0}\limsup_{n \to \infty} \Pr(d_{p}^{M_{1}}(L_{n}^{(0)},L_{n}^{(u)}) > \epsilon)=0$$
for every $\epsilon >0$, from (\ref{e:mainconv}) and (\ref{e:mainconv2}) by a generalization of Slutsky's theorem (see Theorem 3.5 in Resnick~\cite{Resnick07}) it will follow that
$ L_{n}^{(0)} \dto L^{(0)}$ as $n \to \infty$,
in $D^{1} \times D^{2}_{\uparrow}$ with the weak $M_{1}$ topology.
Recalling the definitions, the fact that the metric $d_{p}^{M_{1}}$ is bounded above by the uniform metric (see Theorem 12.10.3 in Whitt~\cite{Whitt02}), and using stationarity and Markov's inequality we obtain
 \begin{eqnarray}\label{e:slutsky}
    \nonumber  \Pr ( d_{p}^{M_{1}}(L_{n}^{(0)}, L_{n}^{(u)}) > \epsilon ) & \leq & \Pr \bigg(
     \sup_{t \in [0,1]} \|L_{n}^{(0)}(t) - L_{n}^{(u)}(t)\| >  \epsilon \bigg) \\[0.3em]
   \nonumber & \hspace*{-16em} = & \ \hspace*{-8em} \Pr \bigg[
       \sup_{t \in [0,1]}  \bigg| \sum_{i=1}^{\lfloor nt \rfloor} \frac{Z_{i}}{a_{n}}
       1_{ \big\{ \frac{|Z_{i}|}{a_{n}} \leq u \big\} } - \lfloor nt \, \rfloor \mathrm{E} \Big( \frac{Z_{1}}{a_{n}} 1_{ \big\{ \frac{|Z_{1}|}{a_{n}} \leq u \big\} } \Big) \bigg|  > \epsilon
       \bigg]\\[0.4em]
    \nonumber & \hspace*{-16em} = & \ \hspace*{-8em} \Pr \bigg[
       \max_{1 \leq k \leq n}  \bigg| \sum_{i=1}^{k} \bigg( \frac{Z_{i}}{a_{n}}
       1_{ \big\{ \frac{|Z_{i}|}{a_{n}} \leq u \big\} } - \mathrm{E} \Big( \frac{Z_{1}}{a_{n}} 1_{ \big\{ \frac{|Z_{1}|}{a_{n}} \leq u \big\} } \Big) \bigg) \bigg|  > \epsilon
       \bigg]
 \end{eqnarray}
Therefore we have to show
\begin{equation}\label{e:slutskycondition}
 \lim_{u \to 0}\limsup_{n \to \infty} \Pr \bigg[
       \max_{1 \leq k \leq n}  \bigg| \sum_{i=1}^{k} \bigg( \frac{Z_{i}}{a_{n}}
       1_{ \big\{ \frac{|Z_{i}|}{a_{n}} \leq u \big\} } - \mathrm{E} \Big( \frac{Z_{1}}{a_{n}} 1_{ \big\{ \frac{|Z_{1}|}{a_{n}} \leq u \big\} } \Big) \bigg) \bigg|  > \epsilon
       \bigg].
\end{equation}
  For $\alpha \in [1,2)$ this relation is simply condition (\ref{e:vsvcond}). In the case $\alpha \in (0,1)$, let
$$ I(u,n,\epsilon) = \Pr \bigg[
       \max_{1 \le k \le n}  \bigg| \sum_{i=1}^{k} \bigg( \frac{Z_{i}}{a_{n}} \,
       1_{ \big\{ \frac{|Z_{i}|}{a_{n}} \le u \big\} } -  \E \bigg( \frac{Z_{1}}{a_{n}} \,
       1_{ \big\{ \frac{|Z_{1}|}{a_{n}} \le u \big\} } \bigg) \bigg) \bigg| > \epsilon
       \bigg].$$
 Using stationarity and Chebyshev's inequality we get the bound
 \begin{eqnarray}\label{e:alpha01}
   \nonumber I(u,n,\epsilon) & \le & \epsilon^{-1} \E \bigg[ \sum_{i=1}^{n} \bigg| \frac{Z_{i}}{a_{n}} \,
       1_{ \big\{ \frac{|Z_{i}|}{a_{n}} \le u \big\} } -  \E \bigg( \frac{Z_{1}}{a_{n}} \,
       1_{ \big\{ \frac{|Z_{1}|}{a_{n}} \le u \big\} } \bigg) \bigg| \bigg]\\[0.5em]
     \nonumber    &\leq& \frac{2n}{\epsilon} \E \bigg( \frac{|Z_{1}|}{a_{n}} \, 1_{ \big\{ \frac{|Z_{1}|}{a_{n}}
            \le u \big\} } \bigg)\\[0.5em]
  \nonumber & = & \ \frac{2u}{\epsilon} \cdot n \Pr (|Z_{1}| > a_{n}) \cdot \frac{\Pr(|Z_{1}| > ua_{n})}{\Pr(|Z_{1}|>a_{n})} \cdot
          \frac{\mathrm{E}(|Z_{1}| 1_{ \{ |Z_{1}| \leq ua_{n} \} })}{ ua_{n} \Pr (|Z_{1}| >ua_{n})}.
 \end{eqnarray}
 Since the random variable $Z_{1}$ is regularly varying with index $\alpha$, it holds that
  $$ \lim_{n \to \infty} \frac{\Pr(|Z_{1}| > ua_{n})}{\Pr(|Z_{1}|>a_{n})} = u^{-\alpha}.$$
  By Karamata's theorem
  $$ \lim_{n \to \infty} \frac{\mathrm{E}(|Z_{1}| 1_{ \{ |Z_{1}| \leq ua_{n} \} })}{ ua_{n} \Pr (|Z_{1}| >ua_{n})} = \frac{\alpha}{1-\alpha},$$
 and therefore taking into account relation (\ref{eq:niz}), we get
 $$ \limsup_{n \to \infty}  I(u,n,\epsilon) \leq u^{1-\alpha} \frac{\alpha}{\epsilon (1-\alpha)}.$$
 Since in this case $1-\alpha >0$, letting $u \to 0$ we obtain
 $$ \lim_{u \to 0}\limsup_{n \to \infty} I(u, n, \epsilon)=0.$$
 Hence we conclude
 \begin{equation}\label{e:weakM1L0}
  L_{n}^{(0)}(\,\cdot\,) \dto L^{(0)}(\,\cdot\,) \qquad \textrm{as} \ n \to \infty,
  \end{equation}
in $D^{1} \times D^{2}_{\uparrow}$ with the weak $M_{1}$ topology.

By Karamata's theorem, as $n \to \infty$,
\begin{eqnarray*}
  n\,\mathrm{E} \Big( \frac{Z_{1}}{a_{n}} 1_{\{|Z_{1}| \leq a_{n} \}} \Big) \to (p-r)\frac{\alpha}{1-\alpha}, && \textrm{if}  \ \ \alpha \in (0,1),\\[0.5em]
  n\,\mathrm{E} \Big( \frac{Z_{1}}{a_{n}} 1_{\{ |Z_{1}| > a_{n} \}} \Big) \to (p-r)\frac{\alpha}{\alpha-1}, && \textrm{if} \ \ \alpha \in (1,2),
\end{eqnarray*}
with $p$ and $r$ as in (\ref{eq:pq}). Therefore conditions (\ref{e:oceknula}) and (\ref{e:sim}), and Lemma~\ref{l:weakM1transf}, applied to the convergence in (\ref{e:weakM1L0}), yield
\begin{equation}\label{eq:weakconvLn*}
 L_{n}^{*}(\,\cdot\,) \dto L(\,\cdot\,) \qquad \textrm{as} \  n \to \infty
\end{equation}
in $D^{1} \times D^{2}_{\uparrow}$ with the weak $M_{1}$ topology, where
$$L_{n}^{*}(\,\cdot\,) := \bigg( \sum_{i = 1}^{\lfloor n \, \cdot \, \rfloor} \frac{Z_{i}}{a_{n}}, \bigvee_{i=1}^{\floor{n\,\cdot}} \frac{|Z_{i}|}{a_{n}} 1_{ \{ Z_{i} > 0 \} }, \bigvee_{i=1}^{\floor{n\,\cdot}}\frac{|Z_{i}|}{a_{n}} 1_{\{ Z_{i}<0 \}}  \bigg)
    $$
and $L:=(V, M^{(1)}, M^{(2)})$, with
$$V(t)= \left\{ \begin{array}{cc}
                                   V^{(0)}(t), & \quad \alpha = 1,\\[0.4em]
                                   V^{(0)}(t) + (p-r)\frac{\alpha}{1-\alpha}, & \quad \alpha \in (0,1) \cup (1,2),
                                 \end{array}\right.$$
being an $\alpha$--stable L\'{e}vy process
 with characteristic triple
$(0,\mu,0)$ if $\alpha=1$ and $(0,\mu,(p-r)\alpha/(1-\alpha))$ if $\alpha \in (0,1) \cup (1,2)$. Note that in case $\alpha=1$ it holds that $L_{n}^{*}=L_{n}^{(0)}$ (since $Z_{1}$ is symmetric).

 It is known that $D^{1}$ equipped with the Skorokhod $J_{1}$ topology is a Polish space, i.e.~metrizable as a complete separable metric space (see Billingsley~\cite{Bi68}, Section 14), and therefore the same holds for the $M_{1}$ topology, since it is topologically complete (see Whitte~\cite{Whitt02}, Section 12.8) and separability remains preserved in the weaker topology. Since the space $D^{1}_{\uparrow}$ is a closed subspace of $D^{1}$ (cf.~Whitt~\cite{Whitt02}, Lemma 13.2.3), it is also Polish. Further, the space $D^{1} \times D^{2}_{\uparrow}$ equipped with the weak $M_{1}$ topology is separable as a direct product of three separable topological spaces. It is also topologically complete since the product metric $d_{p}^{M_{1}}$ inherits the completeness of the component metrics. Hence $D^{1} \times D^{2}_{\uparrow}$ with the weak $M_{1}$ topology is also a Polish space. This allows us to apply Corollary 5.18 in Kallenberg~\cite{Ka97} to conclude that there exists a random vector ($C^{(0)}, C^{(1)}, C^{(2)})$, independent of $(V, M^{(1)}, M^{(2)})$, such that
\begin{equation}\label{e:eqdistrC}
(C^{(0)}, C^{(1)}, C^{(2)}) \eind (C, C_{+}, C_{-}).
\end{equation}
This, relation (\ref{eq:weakconvLn*}), the fact that $(C, C_{+}, C_{-})$ is independent of $L_{n}^{*}$, and Theorem 3.29 in Kallenberg~\cite{Ka97}, imply that, as $n \to \infty$,
  \begin{equation}\label{e:zajedkonvK1}
   (B^{0}, B^{1}, B^{2}, L_{n}^{*1}, L_{n}^{*2}, L_{n}^{*3}) \dto (B^{(0)}, B^{(1)}, B^{(2)}, V, M^{(1)}, M^{(2)})
  \end{equation}
  in $D^{1} \times D^{2}_{\uparrow} \times D^{1} \times D^{2}_{\uparrow}$ with the product $M_{1}$ topology, where $L_{n}^{*i}$ is the $i$--th component of $L_{n}^{*}$ $(i=1,2,3)$, $B^{0}(t)=C$, $B^{1}(t)=C_{+}$, $B^{2}(t)=C_{-}$, $B^{(0)}(t)=C^{(0)}$, $B^{(1)}(t)=C^{(1)}$ and $B^{(2)}(t)=C^{(2)}$ for $t \in [0,1]$.

Let $g \colon D^{1} \times D^{2}_{\uparrow} \times D^{1} \times D^{2}_{\uparrow} \to D^{1} \times D^{2}_{\uparrow}$ be a function defined by
$$ g(x) = (x_{1}  x_{4}, x_{2}  x_{5}, x_{3}  x_{6}), \qquad x=(x_{1}, \ldots, x_{6}) \in D^{1} \times D^{2}_{\uparrow} \times D^{1} \times D^{2}_{\uparrow}.$$
 Denote by $\widetilde{D}_{1,2,3}$ the set of all functions in $D^{1} \times D^{2}_{\uparrow} \times D^{1} \times D^{2}_{\uparrow}$ for which the first three component functions have no discontinuity points, that is
$$ \widetilde{D}_{1,2,3} = \{ (x_{1}, \ldots, x_{6}) \in D^{1} \times D^{2}_{\uparrow} \times D^{1} \times D^{2}_{\uparrow} : \textrm{Disc}(x_{i}) = \emptyset, \ i=1,2,3 \}.$$
By Lemma~\ref{l:contmultpl} the function $g$ is continuous on the set $\widetilde{D}_{1,2,3}$ in the weak $M_{1}$ topology, and hence $\textrm{Disc}(g) \subseteq \widetilde{D}_{1,2,3}^{c}$. Denoting
$\widetilde{D}_{1} = \{ u \in D^{1}_{\uparrow} : \textrm{Disc}(u)= \emptyset \}$ we obtain
\begin{eqnarray*}
\Pr[ (B^{(0)}, B^{(1)}, B^{(2)}, V, M^{(1)}, M^{(2)}) \in \textrm{Disc}(g) ] & & \\[0.6em]
& \hspace*{-26em} \leq &  \hspace*{-13em} \Pr [ (B^{(0)}, B^{(1)}, B^{(2)}, V, M^{(1)}, M^{(2)})  \in \widetilde{D}_{1,2,3}^{c}]\\[0.6em]
 &  \hspace*{-26em} \leq & \hspace*{-13em} \Pr [ \{ B^{(0)} \in \widetilde{D}_{1}^{c} \} \cup \{ B^{(1)} \in \widetilde{D}_{1}^{c} \} \cup \{ B^{(2)} \in \widetilde{D}_{1}^{c} \}]=0,
 \end{eqnarray*}
where the last equality holds since $B^{(0)}$, $B^{(1)}$ and $B^{(2)}$ have no discontinuity points. This allows us to apply the continuous mapping theorem to relation (\ref{e:zajedkonvK1}) yielding
$g(B^{0}, B^{1}, B^{2}, L_{n}^{*1}, L_{n}^{*2}, L_{n}^{*3}) \dto g(B^{(0)}, B^{(1)}, B^{(2)}, V, M^{(1)}, M^{(2)})$ as $n \to \infty$, that is
\begin{eqnarray}\label{e:zajedkonvK2}
 \nonumber \bigg( \sum_{i = 1}^{\lfloor n \, \cdot \, \rfloor} \frac{CZ_{i}}{a_{n}}, \bigvee_{i=1}^{\floor{n\,\cdot}} \frac{C_{+}|Z_{i}|}{a_{n}} 1_{ \{ Z_{i} > 0 \} }, \bigvee_{i=1}^{\floor{n\,\cdot}}\frac{C_{-}|Z_{i}|}{a_{n}} 1_{\{ Z_{i}<0 \}}  \bigg) &&\\[0.8em]
 &\hspace*{-32em} \dto & \hspace*{-16em} (C^{(0)}V(\,\cdot\,), C^{(1)} M^{(1)}(\,\cdot\,), C^{(2)} M^{(2)}(\,\cdot\,))
\end{eqnarray}
 in $D^{1} \times D^{2}_{\uparrow}$ with the weak $M_{1}$ topology.

 Lemma~\ref{l:weakM2transf} implies the function $f \colon D^{1} \times D^{2}_{\uparrow} \to D^{1} \times D^{1}_{\uparrow}$, defined by
 $$f(x,y,z) = (x, y \vee z), \qquad (x,y,z) \in D^{1} \times D^{2}_{\uparrow},$$
 is continuous when $D^{1} \times D^{2}_{\uparrow}$ and $D^{1} \times D^{1}_{\uparrow}$ are endowed with the weak $M_{1}$ topology. Therefore
 from (\ref{e:zajedkonvK2}), by an application of the continuous mapping theorem, we obtain, as $n \to \infty$,
\begin{eqnarray}\label{e:zajedkonvK3}
 \nonumber \bigg( \sum_{i = 1}^{\lfloor n \, \cdot \, \rfloor} \frac{CZ_{i}}{a_{n}}, \bigvee_{i=1}^{\floor{n\,\cdot}} \frac{C_{+}|Z_{i}|}{a_{n}} 1_{ \{ Z_{i} > 0 \} } \vee \bigvee_{i=1}^{\floor{n\,\cdot}}\frac{C_{-}|Z_{i}|}{a_{n}} 1_{\{ Z_{i}<0 \}}  \bigg) &&\\[0.8em]
 \nonumber &\hspace*{-32em} \dto & \hspace*{-16em} (C^{(0)}V(\,\cdot\,), C^{(1)} M^{(1)}(\,\cdot\,) \vee C^{(2)} M^{(2)}(\,\cdot\,)),
\end{eqnarray}
that is,
$\widetilde{L}_{n} \dto  (C^{(0)}V, C^{(1)}M^{(1)} \vee C^{(2)}M^{(2)})$
in $D^{1} \times D^{1}_{\uparrow}$ with the weak $M_{1}$ topology, which completes the proof.
\end{proof}

\section{Finite order moving average processes}
\label{S:FiniteMA}

Fix $q \in \mathbb{N}$ and let $C_{0}, C_{1}, \ldots , C_{q}$ be random variables satisfying
\be\label{eq:FiniteMAcond}
0 \le \sum_{i=0}^{s}C_{i} \Bigg/ \sum_{i=0}^{q}C_{i} \le 1 \ \ \textrm{a.s.} \qquad \textrm{for every} \ s=0, 1, \ldots, q.
\ee
In this case $C$, $C_{+}$ and $C_{-}$ reduce to
$$C= \sum_{i=0}^{q}C_{i}, \quad C_{+}= \max_{0 \leq i \leq q} (C_{i} \vee 0) \quad \textrm{and} \quad C_{-}= \max_{0 \leq i \leq q}(-C_{i} \vee 0).$$
Condition (\ref{eq:FiniteMAcond}) implies that $C$,
$ \sum_{i=0}^{s}C_{i}$ and $ \sum_{i=s}^{q}C_{i}$ are a.s.~of the same sign for every $s=0,1,\ldots,q$. If the $C_{j}$'s are all nonnegative or all nonpositive, then condition (\ref{eq:FiniteMAcond}) is trivially satisfied. Recall that $L_{n}=(V_{n}, M_{n})$, where
\begin{equation*}
V_{n}(t) = \frac{1}{a_{n}}  \sum_{i=1}^{\floor {nt}}X_{i}, \qquad t \in [0,1],
\end{equation*}
and
\begin{equation*}
M_{n}(t) = \left\{ \begin{array}{cc}
                                   \displaystyle \frac{1}{a_{n}} \bigvee_{i=1}^{\floor {nt}}X_{i}, & \quad  \displaystyle t \in \Big[ \frac{1}{n}, 1 \Big],\\[1.4em]
                                   \displaystyle \frac{X_{1}}{a_{n}}, & \quad \displaystyle  t \in \Big[0,  \frac{1}{n} \Big\rangle,
                                 \end{array}\right.
\end{equation*}
are the partial sum and partial maxima processes constructed from moving average processes
$$ X_{i} = \sum_{j=0}^{q}C_{j}Z_{i-j}, \qquad i \in \mathbb{Z},$$
with the normalizing sequence $(a_n)$ as in~\eqref{eq:niz}.

\begin{thm}\label{t:FinMA}
Let $(Z_{i})_{i \in \mathbb{Z}}$ be a strictly stationary and strongly mixing sequence of regularly varying random variables satisfying $(\ref{e:regvar})$ and $(\ref{eq:pq})$ with $\alpha \in (0,2)$, such that the local dependence condition $D'$ and conditions $(\ref{e:oceknula})$ and $(\ref{e:sim})$ hold. If $\alpha \in [1,2)$, also suppose that condition $(\ref{e:vsvcond})$ holds. Let
 $C_{0}, \ldots, C_{q}$ be random variables, independent of $(Z_{i})$, such that condition $(\ref{eq:FiniteMAcond})$ holds.
Then, as $n \to \infty$,
\begin{equation*}\label{e:mainconvFinMA}
L_{n}(\,\cdot\,) \dto  (C^{(0)}V(\,\cdot\,), C^{(1)}M^{(1)}(\,\cdot\,) \vee C^{(2)}M^{(2)}(\,\cdot\,))
\end{equation*}
 in $D^{1} \times D^{1}_{\uparrow}$ with the weak $M_{2}$ topology,
where $V$ is an $\alpha$--stable L\'{e}vy process with characteristic triple $(0, \mu, b)$, with $\mu$ as in $(\ref{eq:mu})$,
$$ b = \left\{ \begin{array}{cc}
                                   0, & \quad \alpha = 1,\\[0.4em]
                                   (p-r)\frac{\alpha}{1-\alpha}, & \quad \alpha \in (0,1) \cup (1,2),
                                 \end{array}\right.$$
$M^{(1)}$ and $M^{(2)}$ are extremal processes with exponent measures
$p \alpha x^{-\alpha-1}1_{(0,\infty)}(x)\,dx$ and $r \alpha x^{-\alpha-1}1_{(0,\infty)}(x)\,dx$ respectively, with $p$ and $r$ defined in $(\ref{eq:pq})$, and $(C^{(0)}, C^{(1)}, C^{(2)})$ is a random vector, independent of $(V, M^{(1)}, M^{(2)})$, such that
$(C^{(0)}, C^{(1)}, C^{(2)}) \eind (C, C_{+}, C_{-})$.
\end{thm}

\begin{proof} 
By Proposition~\ref{p:FLT}, $\widetilde{L}_{n} = (\widetilde{V}_{n}, \widetilde{M}_{n}) \dto  (C^{(0)}V, C^{(1)}M^{(1)} \vee C^{(2)}M^{(2)})$ as $n \to \infty$, in $D^{1} \times D^{1}_{\uparrow}$ with the weak $M_{1}$ topology, where
$$ \widetilde{V}_{n}(t) =  \sum_{i=1}^{\floor {nt}}\frac{CZ_{i}}{a_{n}} \quad \textrm{and} \quad \widetilde{M}_{n}(t)= \bigvee_{i=1}^{\floor {nt}}\frac{ |Z_{i}|}{a_{n}}(C_{+}1_{\{ Z_{i} >0 \}} + C_{-} 1_{\{ Z_{i}<0 \}}), \qquad t \in [0,1].$$
 Since $M_{1}$ convergence implies $M_{2}$ convergence, we have
\begin{equation}\label{e:pomkonvFMA}
\widetilde{L}_{n}(\,\cdot\,) = (\widetilde{V}_{n}(\,\cdot\,), \widetilde{M}_{n}(\,\cdot\,)) \dto  (C^{(0)}V(\,\cdot\,), C^{(1)}M^{(1)}(\,\cdot\,) \vee C^{(2)}M^{(2)}(\,\cdot\,))
\end{equation}
 in $D^{1} \times D^{1}_{\uparrow}$ with the weak $M_{2}$ topology as well.

If we show that
$$ \lim_{n \to \infty} \Pr [ d_{p}^{M_{2}}(\widetilde{L}_{n}, L_{n}) > \delta ]=0$$
for any $\delta > 0$, then from (\ref{e:pomkonvFMA}) by an application of Slutsky's theorem (see Theorem 3.4 in Resnick~\cite{Resnick07}) it will follow that $L_{n} \dto (C^{(0)}V, C^{(1)}M^{(1)} \vee C^{(2)}M^{(2)})$ as $n \to \infty$, in $D^{1} \times D^{1}_{\uparrow}$ with the weak $M_{2}$ topology. By the definition of the metric $d_{p}^{M_{2}}$ in (\ref{e:defdp}) it suffices to show that
\begin{equation}\label{e:SlutskyV}
\lim_{n \to \infty} \Pr [ d_{M_{2}}(\widetilde{V}_{n}, V_{n}) > \delta ]=0
\end{equation}
and
\begin{equation}\label{e:SlutskyM}
\lim_{n \to \infty} \Pr [ d_{M_{2}}(\widetilde{M}_{n}, M_{n}) > \delta ]=0.
\end{equation}
Relation (\ref{e:SlutskyV}) was established in the proof of Theorem 2.1 in Krizmani\'{c}~\cite{Kr22-1}. As for relation (\ref{e:SlutskyM}), it holds in a special case when the innovations $Z_{i}$ are i.i.d.~(see the proof of Theorem 3.3 in Krizmani\'{c}~\cite{Kr22-2}). It remains to show that it also holds in the weak dependence case, i.e.~when the independence property is replaced by the local dependence condition $D'$ in (\ref{e:D'cond}). Krizmani\'{c}~\cite{Kr22-2} showed that
\begin{equation}\label{e:H123}
 \{d_{M_{2}}(\widetilde{M}_{n}, M_{n}) > \delta \} \subseteq H_{n, 1} \cup H_{n, 2} \cup H_{n, 3},
\end{equation}
where
\begin{eqnarray}
 \nonumber H_{n, 1} & = & \bigg\{ \exists\,l \in \{-q,\ldots,q\} \cup \{n-q+1, \ldots, n\} \ \textrm{such that} \ \frac{ C_{*}|Z_{l}|}{a_{n}} > \frac{\delta}{4(q+1)} \bigg\},\\[0.2em]
  \nonumber H_{n, 2} & = & \bigg\{ \exists\,k \in \{1, \ldots, n\} \ \textrm{and} \ \exists\,l \in \{k-q,\ldots,k+q\} \setminus \{k\} \ \textrm{such that}\\[0.2em]
  \nonumber & & \ \frac{ C_{*}|Z_{k}|}{a_{n}} > \frac{\delta}{4(q+1)} \ \textrm{and} \  \frac{ C_{*}|Z_{l}|}{a_{n}} > \frac{\delta}{4(q+1)}  \bigg\},\\[0.2em]
  \nonumber H_{n, 3} & = & \bigg\{ \exists\,k \in \{1, \ldots, n\}, \ \exists\,j \in \{1,\ldots,n\} \setminus \{k,\ldots,k+q\}, \ \exists\,l_{1} \in \{0,\ldots,q\}\\[0.2em]
  \nonumber & & \textrm{and} \ \exists\,l \in \{0,\ldots,q\} \setminus \{l_{1}\} \ \textrm{such that} \ \frac{ C_{*}|Z_{k}|}{a_{n}} > \frac{\delta}{4(q+1)},\\[0.2em]
  \nonumber & & \ \frac{ C_{*}|Z_{j-l_{1}}|}{a_{n}} > \frac{\delta}{4(q+1)} \ \textrm{and} \  \frac{ C_{*}|Z_{j-l}|}{a_{n}} > \frac{\delta}{4(q+1)}  \bigg\},
\end{eqnarray}
with $C_{*}:= C_{+} \vee C_{-}$, and the independence property was used there only in establishing
$$ \lim_{n \to \infty} \Pr(H_{n,2})=0 \qquad \textrm{and} \qquad \lim_{n \to \infty} \Pr(H_{n,3})=0.$$
 Now we are going to show the last two limiting relations still hold under condition $D'$.
For an arbitrary $M>0$ it holds that
\begin{eqnarray}
 \nonumber \Pr(H_{n, 2} \cap \{C_{*} \leq M\}) &&\\[0.6em]
  \nonumber & \hspace*{-12em} = & \hspace*{-6em} \sum_{k=1}^{n} \sum_{\scriptsize \begin{array}{c}
                          l=k-q  \\[0.1em]
                          l \neq k
                        \end{array}}^{k+q} \Pr \bigg( \frac{C_{*}|Z_{k}|}{a_{n}} > \frac{\delta}{4(q+1)},\,\frac{C_{*}|Z_{l}|}{a_{n}} > \frac{\delta}{4(q+1)},\,C_{*} \leq M \bigg)\\[0.2em]
   \nonumber &  \hspace*{-12em} \leq & \hspace*{-6em} 2n \sum_{i=1}^{q}
    \Pr \bigg( \frac{|Z_{0}|}{a_{n}} > \frac{\delta}{4(q+1)M},\,\frac{|Z_{i}|}{a_{n}} > \frac{\delta}{4(q+1)M} \bigg)\\[0.5em]
  \nonumber &  \hspace*{-12em} \leq & \hspace*{-6em} 2n \sum_{i=1}^{\lfloor n/k \rfloor}
    \Pr \bigg( \frac{|Z_{0}|}{a_{n}} > \frac{\delta}{4(q+1)M},\,\frac{|Z_{i}|}{a_{n}} > \frac{\delta}{4(q+1)M} \bigg)
 \end{eqnarray}
 for all positive integers $k$ such that $k \leq n/q$. By letting $n \to \infty$, and then $k \to \infty$, we see that condition $D'$ yields that
 $\Pr(H_{n, 2} \cap \{ C \leq M \}) \to 0$ as $n \to \infty$. Hence
 $$ \limsup_{n \to \infty}  \Pr(H_{n, 2}) \leq \limsup_{n \to \infty} \Pr(H_{n, 2} \cap \{C_{*} > M\}) \leq \Pr (C_{*} > M),$$
and letting $M \to \infty$ we conclude
\begin{equation}\label{e:est5}
\lim_{n \to \infty} \Pr(H_{n, 2})=0.
\end{equation}
Note that
\begin{eqnarray}
  \nonumber H_{n, 3} & \subseteq & \bigg\{ \exists\,s \in \{1-q, \ldots, n\} \ \textrm{and} \ \exists\,s_{1} \in \{s-q,\ldots,s+q\} \setminus \{s\} \\[0.2em]
  \nonumber & & \textrm{such that} \ \frac{ C_{*}|Z_{s}|}{a_{n}} > \frac{\delta}{4(q+1)} \ \textrm{and} \  \frac{ C_{*}|Z_{s_{1}}|}{a_{n}} > \frac{\delta}{4(q+1)}  \bigg\},
\end{eqnarray}
which implies
\begin{eqnarray}
 \nonumber \Pr(H_{n, 3} \cap \{C_{*} \leq M\}) &&\\[0.6em]
  \nonumber & \hspace*{-12em} = & \hspace*{-6em} \sum_{s=1-q}^{n} \sum_{\scriptsize \begin{array}{c}
                          s_{1}=s-q  \\[0.1em]
                          s_{1} \neq s
                        \end{array}}^{s+q} \Pr \bigg( \frac{C_{*}|Z_{s}|}{a_{n}} > \frac{\delta}{4(q+1)},\,\frac{C_{*}|Z_{s_{1}}|}{a_{n}} > \frac{\delta}{4(q+1)},\,C_{*} \leq M \bigg)\\[0.2em]
   \nonumber &  \hspace*{-12em} \leq & \hspace*{-6em} 2[n-(1-q)+1] \sum_{i=1}^{q}
    \Pr \bigg( \frac{|Z_{0}|}{a_{n}} > \frac{\delta}{4(q+1)M},\,\frac{|Z_{i}|}{a_{n}} > \frac{\delta}{4(q+1)M} \bigg)\\[0.5em]
  \nonumber &  \hspace*{-12em} \leq & \hspace*{-6em} 4n \sum_{i=1}^{\lfloor n/k \rfloor}
    \Pr \bigg( \frac{|Z_{0}|}{a_{n}} > \frac{\delta}{4(q+1)M},\,\frac{|Z_{i}|}{a_{n}} > \frac{\delta}{4(q+1)M} \bigg)
 \end{eqnarray}
 for arbitrary $M>0$, $n \geq q$ and positive integers $k$ such that $k \leq n/q$. Hence, similar as in (\ref{e:est5}), we obtain
 \begin{equation}\label{e:est10}
\lim_{n \to \infty} \Pr(H_{n, 3})=0.
\end{equation}
Further, by stationarity and the regular variation property it holds that
$$\limsup_{n \to \infty} \Pr (H_{n,1} \cap \{ C_{*} \leq M \}) \leq (3q+1) \limsup_{n \to \infty} \Pr \bigg( \frac{|Z_{1}|}{a_{n}} > \frac{\delta}{4(q+1)M} \bigg)=0$$
for arbitrary $M>0$,
and this, as before, implies
\begin{equation}\label{e:est2}
\lim_{n \to \infty} \Pr(H_{n, 1})=0.
\end{equation}
Now, from (\ref{e:H123})--(\ref{e:est2}) we obtain (\ref{e:SlutskyM}), and therefore finally conclude that  $L_{n} \dto (C^{(0)}V, C^{(1)}M^{(1)} \vee C^{(2)}M^{(2)})$ as $n \to \infty$, in $D^{1} \times D^{1}_{\uparrow}$ with the weak $M_{2}$ topology.
\end{proof}

\begin{rem}\label{r:jointdepend1}
From the proof of Proposition~\ref{p:FLT} it follows that the components of the limiting process $(C^{(0)}V, C^{(1)}M^{(1)} \vee C^{(2)}M^{(2)})$ can be expressed as functionals of the limiting point process $N = \sum_{i}  \delta_{(T_{i}, P_{i}Q_{i})}$ from relation (\ref{e:BaTa1}), that is
$$ V(t) =  \lim_{u \to 0} \bigg( \sum_{T_{i} \le t} P_{i}Q_{i} 1_{\{ P_{i} > u \}} - t  \int_{u < |x| \leq 1} x\mu(\rmd x) \bigg) + (p-r)\frac{\alpha}{1-\alpha} 1_{\{\alpha \neq 1\}},$$
where the limit holds almost surely uniformly on $[0,1]$, and
$$  M^{(1)}(t) = \bigvee_{T_{i} \leq t} P_{i} 1_{\{ Q_{i}>0 \}}, \qquad  M^{(2)}(t) = \bigvee_{T_{i} \leq t} P_{i} 1_{\{ Q_{i}<0 \}}.$$
\end{rem}

\begin{rem}\label{r:M2M1}
Theorem~\ref{t:FinMA} establishes functional convergence of the joint stochastic process $L_{n}$ of partial sums and maxima in the space $D^{1} \times D^{1}_{\uparrow}$ endowed with the weak $M_{2}$ topology induced by the metric $d_{p}^{M_{2}}$ given in (\ref{e:defdp}). Since for the second coordinate of $L_{n}$, i.e.~the partial maxima process $M_{n}$, functional convergence holds also in the stronger $M_{1}$ topology, it is possible to obtain a sort of joint convergence of $L_{n}$ in the $M_{2}$ topology on the first coordinate and in the $M_{1}$ topology on the second coordinate.

Precisely, by Remark 12.8.1 in Whitt~\cite{Whitt02} the following metric is a complete metric topologically equivalent to $d_{M_{1}}$:
$$ {d_{M_{1}}^{*}}(x_{1}, x_{2}) = d_{M_{2}}(x_{1}, x_{2}) + \lambda (\widehat{\omega}(x_{1},\cdot), \widehat{\omega}(x_{2},\cdot)),$$
where $\lambda$ is the L\'{e}vy metric on a space of distributions
$$ \lambda (F_{1},F_{2}) = \inf \{ \epsilon >0 : F_{2}(x-\epsilon) - \epsilon \leq F_{1}(x) \leq F_{2}(x+\epsilon) + \epsilon \ \ \textrm{for all} \ x \}$$
and
$$ \widehat{\omega}(x,z) = \left\{ \begin{array}{cc}
                                   \omega(x,e^{z}), & \quad z<0,\\[0.4em]
                                   \omega(x,1), & \quad z \geq 0,
                                 \end{array}\right.$$
with
$ \omega (x,\rho) = \sup_{0 \leq t \leq 1} \omega (x, t, \rho)$ and
\begin{equation*}\label{e:oscillationf}
\omega(x,t,\rho) = \sup_{0 \vee (t-\rho) \leq t_{1} < t_{2} < t_{3} \leq (t+\rho) \wedge 1} ||x(t_{2}) - [x(t_{1}), x(t_{3})]||
\end{equation*}
where $\rho>0$ and $\|z-A\|$ denotes the distance between a point $z$ and a subset $A \subseteq \mathbb{R}$.
Since $\widetilde{M}_{n}$ and $M_{n}$ are nondecreasing, for $t_{1} < t_{2} < t_{3}$ it holds that
$ \|\widetilde{M}_{n}(t_{2}) - [\widetilde{M}_{n}(t_{1}), \widetilde{M}_{n}(t_{3})] \|=0$, which implies $\omega(\widetilde{M}_{n}, \rho)=0$ for all $\rho>0$, and similarly $\omega(M_{n}, \rho)=0$. Therefore $\lambda (\widetilde{M}_{n}, M_{n})=0$, and $d_{M_{1}}^{*}(\widetilde{M}_{n}, M_{n}) = d_{M_{2}}(\widetilde{M}_{n}, M_{n})$.

Now from (\ref{e:SlutskyM}) we obtain
$$ \lim_{n \to \infty} \Pr[d_{M_{1}}^{*}(\widetilde{M}_{n}, M_{n}) > \delta]=0$$
for any $\delta >0$,
which allows us to conclude that $L_{n}$ converges in distribution to $(C^{(0)}V, C^{(1)}M^{(1)} \vee C^{(2)}M^{(2)})$
in the topology induced by the metric
$$ {d_{p}}^{*}(x,y)= \max \{ d_{M_{2}}(x_{1},y_{1}), d_{M_{1}}^{*}(x_{2}, y_{2}) \}$$
 for $x=(x_{1}, x_{2}), y=(y_{1}, y_{2}) \in D^{2}$, that is, in the $M_{2}$ topology on the first coordinate and in the $M_{1}$ topology on the second coordinate.
\end{rem}

\begin{rem}
In the case of deterministic coefficients of the same sign, functional convergence of $L_{n}$ in Theorem~\ref{t:FinMA} can be obtained also by an application of Theorem 3.4 in Krizmani\'{c}~\cite{Kr20} with some appropriate modifications due to different centering and normalizing sequences used. In this case the convergence actually holds in the weak $M_{1}$ topology.
\end{rem}

\section{Infinite order moving average processes}
\label{S:InfiniteMA}

When dealing with infinite order moving averages the standard idea is to approximate them by a sequence of finite order moving averages, for which the weak convergence holds, and to show that the error of approximation is negligible in the limit. In our case, we will approximate them by finite order moving averages for which Theorem~\ref{t:FinMA} holds, and then we will show that the error of approximation is negligible with respect to the uniform metric.


\begin{thm}\label{t:InfMA}
Let $(X_{i})$ be a moving average process defined by
$$ X_{i} = \sum_{j=0}^{\infty}C_{j}Z_{i-j}, \qquad i \in \mathbb{Z},$$
where $(Z_{i})_{i \in \mathbb{Z}}$ is a strictly stationary and strongly mixing sequence of regularly varying random variables satisfying $(\ref{e:regvar})$ and $(\ref{eq:pq})$ with $\alpha \in (0,2)$, such that the local dependence condition $D'$ and conditions $(\ref{e:oceknula})$ and $(\ref{e:sim})$ hold. Let $(C_{i})_{i \geq 0}$ be a sequence of random variables,independent of $(Z_{i})$, satisfying conditions $(\ref{e:momcond})$ and $(\ref{eq:InfiniteMAcond})$. If $\alpha \in (0,1)$ suppose further
\begin{equation}\label{e:mod1}
 \sum_{i=0}^{\infty}\mathrm{E}|C_{i}|^{\gamma} < \infty \qquad \textrm{for some} \ \gamma \in (\alpha, 1),
\end{equation}
while if $\alpha \in [1,2)$ suppose condition $(\ref{e:vsvcond})$ holds,
\begin{equation}\label{eq:infmaTK}
\limsup_{n \to \infty} \sup_{j \geq 0} \mathrm{E} \bigg[ \max_{1 \leq l \leq n} \bigg| \frac{1}{a_{n}} \sum_{i=1}^{l}Z_{i-j} 1_{\{ |Z_{i-j}| \leq a_{n} \}} \bigg|^{r}\bigg] < \infty \qquad \textrm{for some} \ r \geq 1,
\end{equation}
and
\begin{equation}\label{eq:infmaTK3}
\sum_{j=0}^{\infty} \mathrm{E}|C_{j}| < \infty.
\end{equation}
Then, as $n \to \infty$,
\begin{equation}\label{e:mainconvInfMA}
L_{n}(\,\cdot\,) \dto  (C^{(0)}V(\,\cdot\,), C^{(1)}M^{(1)}(\,\cdot\,) \vee C^{(2)}M^{(2)}(\,\cdot\,))
\end{equation}
 in $D^{1} \times D^{1}_{\uparrow}$ with the weak $M_{2}$ topology,
where $V$ is an $\alpha$--stable L\'{e}vy process with characteristic triple $(0, \mu, b)$, with $\mu$ as in $(\ref{eq:mu})$,
$$ b = \left\{ \begin{array}{cc}
                                   0, & \quad \alpha = 1,\\[0.4em]
                                   (p-r)\frac{\alpha}{1-\alpha}, & \quad \alpha \in (0,1) \cup (1,2),
                                 \end{array}\right.$$
$M^{(1)}$ and $M^{(2)}$ are extremal processes with exponent measures
$p \alpha x^{-\alpha-1}1_{(0,\infty)}(x)\,dx$ and $r \alpha x^{-\alpha-1}1_{(0,\infty)}(x)\,dx$ respectively, with $p$ and $r$ defined in $(\ref{eq:pq})$, and $(C^{(0)}, C^{(1)}, C^{(2)})$ is a random vector, independent of $(V, M^{(1)}, M^{(2)})$, such that
$(C^{(0)}, C^{(1)}, C^{(2)}) \eind (C, C_{+}, C_{-})$, where $C=\sum_{i=0}^{\infty}C_{i}$ and $C_{+}, C_{-}$ as defined in $(\ref{e:Cplusminus})$.
\end{thm}

\begin{proof}
For $q \in \mathbb{N}$ define
$$ X_{i}^{q} = \sum_{j=0}^{q-1}C_{j}Z_{i-j} + C^{q} Z_{i-q}, \qquad i \in \mathbb{Z},$$
where $C^{q}= \sum_{j=q}^{\infty}C_{j}$, and let
$$ V_{n, q}(t) = \sum_{i=1}^{\floor{nt}} \frac{X_{i}^{q}}{a_{n}} \quad \textrm{and} \quad M_{n, q}(t) = \frac{X_{1}^{q}}{a_{n}}1_{[0, 1/n \rangle}(t) + \bigvee_{i=1}^{\floor {nt}}\frac{X_{i}^{q}}{a_{n}} 1_{[1/n,1]}(t), \qquad t \in [0,1].$$
Condition (\ref{eq:InfiniteMAcond}) implies that coefficients $C_{0}, \ldots, C_{q-1}, C^{q}$ satisfy condition (\ref{eq:FiniteMAcond}), and hence an application of Theorem~\ref{t:FinMA} to the finite order moving average process $(X_{i}^{q})_{i}$ yields that
\begin{equation}\label{e:SlutskyINFMA1}
 L_{n, q}(\,\cdot\,) := (V_{n, q}(\,\cdot\,), M_{n, q}(\,\cdot\,))  \dto L^{q}(\,\cdot\,) \qquad \textrm{as} \ n \to \infty,
 \end{equation}
in $D^{1} \times D^{1}_{\uparrow}$ with the weak $M_{2}$ topology, with
$L^{q} =(C^{(0)}V, C^{(1)}_{q}M^{(1)} \vee C^{(2)}_{q}M^{(2)})$, where $V$ is an $\alpha$--stable L\'{e}vy process and $M^{(1)}, M^{(2)}$ are extremal processes as in Theorem~\ref{t:FinMA}, and $(C^{(0)}, C^{(1)}_{q}, C^{(2)}_{q})$ is a random vector, independent of $(V, M^{(1)}, M^{(2)})$, such that
$(C^{(0)}, C^{(1)}_{q}, C^{(2)}_{q}) \eind (C, C_{+,q}, C_{-,q})$, with
$$C_{+,q} = \max \{ C_{j} \vee 0 : j =0,\ldots, q-1\} \vee (C^{q} \vee 0)$$
and
 $$C_{-,q} = \max \{ -C_{j} \vee 0 : j =0,\ldots, q-1\} \vee (-C^{q} \vee 0).$$
Since $\sum_{j=0}^{\infty}|C_{j}| < \infty$ a.s.~(which follows from condition (\ref{e:momcond})), it is straightforward to obtain
$$ C_{+,q} \to C_{+} \quad \textrm{and}  \quad  C_{-,q} \to C_{-}$$
almost surely as $q \to \infty$. Therefore
$$ \| (B, B_{+,q}, B_{-,q}) - (B, B_{+}, B_{-})\|_{[0,1]} \to 0$$
almost surely as $q \to \infty$, where $B(t)= C$, $B_{+,q}(t)=C_{+,q}$, $B_{-,q}(t)=C_{-,q}$, $B_{+}(t)=C_{+}$ and $B_{-}(t)=C_{-}$ for $t \in [0,1]$. Since uniform convergence implies Skorokhod $M_{1}$ convergence, it follows that
$d_{p}^{M_{1}}( (B, B_{+,q}, B_{-,q}), (B, B_{+}, B_{-})  ) \to 0$
almost surely as $q \to \infty$, and hence, since almost sure convergence implies convergence in distribution, we have
\begin{equation*}\label{e:contmtKall}
(B, B_{+,q}, B_{-,q}) \dto (B, B_{+}, B_{-}) \qquad \textrm{as} \ q \to \infty,
\end{equation*}
in $D^{1} \times D^{2}_{\uparrow}$ with the weak $M_{1}$ topology. An application of Theorem 3.29 and Corollary 5.18 in Kallenberg~\cite{Ka97} yields
\begin{equation}\label{e:contmtKall2}
 (B^{(0)}, B^{(1)}_{q}, B^{(2)}_{q}, V, M^{(1)}, M^{(2)}) \dto (B^{(0)}, B^{(1)}, B^{(2)}, V, M^{(1)}, M^{(2)})
\end{equation}
in $D^{1} \times D^{2}_{\uparrow} \times D^{1} \times D^{2}_{\uparrow}$ with the product $M_{1}$ topology, where $B^{(0)}(t)=C^{(0)}$, $B^{(1)}_{q}(t)=C^{(1)}_{q}$, $B^{(2)}_{q}(t)=C^{(2)}_{q}$, $B^{(1)}(t)=C^{(1)}$ and $B^{(2)}(t)=C^{(2)}$ for $t \in [0,1]$, and $(C^{(0)}, C^{(1)}, C^{(2)})$ is a random vector, independent of $(V, M^{(1)}, M^{(2)})$, such that $(C^{(0)}, C^{(1)}, C^{(2)}) \eqd (C, C_{+}, C_{-})$.

Now, as in the proof of Proposition~\ref{p:FLT}, an application of the continuous mapping theorem to the convergence relation in (\ref{e:contmtKall2}), with the function $g$ given by
$$ g(x) = (x_{1}  x_{4}, x_{2}  x_{5}, x_{3}  x_{6}), \qquad x=(x_{1}, \ldots, x_{6}) \in D^{1} \times D^{2}_{\uparrow} \times D^{1} \times D^{2}_{\uparrow},$$
yields
$$g (B^{(0)}, B^{(1)}_{q}, B^{(2)}_{q}, V, M^{(1)}, M^{(2)}) \dto g(B^{(0)}, B^{(1)}, B^{(2)}, V, M^{(1)}, M^{(2)}),$$
 as $q \to \infty$, that is
 \begin{equation}\label{e:SlutskyINFMA2}
L^{q} = (C^{(0)}, C^{(1)}_{q} M^{(1)} \vee C^{(2)}_{q} M^{(2)}) \dto L := (C^{(0)}V, C^{(1)}M^{(1)} \vee C^{(2)}M^{(2)})
 \end{equation}
in $D^{1} \times D^{1}_{\uparrow}$ with the weak $M_{1}$ topology, and then also with the weak $M_{2}$ topology.

If we show that for every $\epsilon >0$
\begin{equation}\label{e:Slutskyinf01}
 \lim_{q \to \infty} \limsup_{n \to \infty}\Pr[d_{p}^{M_{2}}(L_{n}, L_{n,q})> \epsilon]=0,
\end{equation}
then from relations (\ref{e:SlutskyINFMA1}) and (\ref{e:SlutskyINFMA2}) by a generalization of Slutsky's theorem (Resnick~\cite{Resnick07}, Theorem 3.5) it will follow that $L_{n}(\,\cdot\,) \dto L(\,\cdot\,)$ in $D^{1} \times D^{1}_{\uparrow}$ with the weak $M_{2}$ topology.  By the definition of the metric $d_{p}^{M_{2}}$ it suffices to show that
\begin{equation}\label{e:SlutskyINFV}
\lim_{q \to \infty} \limsup_{n \to \infty} \Pr [ d_{M_{2}}(V_{n}, V_{n,q}) > \epsilon ]=0
\end{equation}
and
\begin{equation}\label{e:SlutskyINFM}
\lim_{q \to \infty} \limsup_{n \to \infty} \Pr [ d_{M_{2}}(M_{n}, M_{n,q}) > \epsilon]=0.
\end{equation}
Since the metric $d_{M_{2}}$ is bounded above by the uniform metric,
$$\Pr [ d_{M_{2}}(M_{n}, M_{n,q}) > \epsilon]  \leq  \Pr \bigg( \sup_{t \in [0,1]}|M_{n}(t) - M_{n, q}(t)|> \epsilon \bigg),$$
and hence by using the same arguments as in the proof of Theorem 3.4 in Krizmani\'{c}~\cite{Kr22-2}, conditions (\ref{e:momcond}), (\ref{e:mod1}) and (\ref{eq:infmaTK3}) imply
$$ \lim_{q \to \infty} \limsup_{n \to \infty} \Pr \bigg( \sup_{t \in [0,1]}|M_{n}(t) - M_{n, q}(t)|> \epsilon \bigg)=0,$$
and hence relation (\ref{e:SlutskyINFM}) holds.
Note that
 $$ \Pr [ d_{M_{2}}(V_{n}, V_{n,q}) > \epsilon]  \leq  \Pr \bigg( \sup_{t \in [0,1]}|V_{n}(t) - V_{n, q}(t)|> \epsilon \bigg)
   \leq  \Pr \bigg( \sum_{i=1}^{n}\frac{|X_{i}-X_{i}^{q}|}{a_{n}} > \epsilon \bigg).$$
In the case $\alpha \in (0,1)$ by repeating the arguments from the proof of Theorem 3.4 in Krizmani\'{c}~\cite{Kr22-2}, from conditions (\ref{e:momcond}) and (\ref{e:mod1}) we obtain
\begin{equation*}\label{e:infSlutsky}
 \lim_{q \to \infty} \limsup_{n \to \infty} \Pr \bigg( \sum_{i=1}^{n}\frac{|X_{i}-X_{i}^{q}|}{a_{n}} > \epsilon \bigg)=0,
\end{equation*}
and relation (\ref{e:SlutskyINFV}) holds.

In the case $\alpha \in [1,2)$ define $Z_{n,j}^{\leq} = a_{n}^{-1} Z_{j} 1_{\{ |Z_{j}| \leq a_{n} \}}$ and $Z_{n,j}^{>} = a_{n}^{-1} Z_{j} 1_{\{ |Z_{j}| > a_{n} \}}$  for $j \in \mathbb{Z}$ and $n \in \mathbb{N}$. Let
$$ \widetilde{C}_{j} = \left\{ \begin{array}{cl}
                                   C_{j}, & \quad \textrm{if} \ j \geq q+1,\\[0.4em]
                                   C_{q}-C^{q}, & \quad \textrm{if} \ j=q.
                                 \end{array}\right.$$
and observe that
\begin{eqnarray*}
\nonumber  V_{n}(t) - V_{n,q}(t) & = & \sum_{i=1}^{\floor{nt}} \frac{1}{a_{n}} \bigg( \sum_{j=q}^{\infty}C_{j}Z_{i-j} -C^{q}Z_{i-q} \bigg) \\[0.4em]
& = &  \sum_{i=1}^{\floor{nt}}  \sum_{j=q}^{\infty} \widetilde{C}_{j}Z_{n,i-j}^{\leq} + \sum_{i=1}^{\floor{nt}}  \sum_{j=q}^{\infty} \widetilde{C}_{j}Z_{n,i-j}^{>}.
\end{eqnarray*}
Therefore
\begin{eqnarray}\label{eq:I1I2}
\nonumber   \Pr[d_{M_{2}}(V_{n}, V_{n,q})> \epsilon] & \leq & \Pr \bigg( \sup_{ t \in [0,1]} |V_{n}(t) - V_{n,q}(t)| > \epsilon \bigg)\\[0.4em]
\nonumber & \hspace*{-18em} \leq & \hspace*{-9em} \Pr \bigg( \max_{1 \leq l \leq n} \bigg| \sum_{i=1}^{l} \sum_{j=q}^{\infty} \widetilde{C}_{j}Z_{n,i-j}^{\leq} \bigg| > \frac{\epsilon}{2} \bigg) + \Pr \bigg( \max_{1 \leq l \leq n} \bigg| \sum_{i=1}^{l} \sum_{j=q}^{\infty} \widetilde{C}_{j}Z_{n,i-j}^{>} \bigg| > \frac{\epsilon}{2} \bigg)\\[0.4em]
& \hspace*{-18em} =: & \hspace*{-9em} I_{1} + I_{2}.
\end{eqnarray}
Now following the arguments from the proof of Theorem 3.1 in Krizmani\'{c}~\cite{Kr22-1} we obtain, by
H\"{o}lder's and Markov's inequalities and the fact that the sequence $(C_{i})_{i \geq 0}$ is independent of $(Z_{i})$,
\begin{equation*}
I_{1} \leq \mathrm{E} \bigg( \sum_{j=q}^{\infty}|\widetilde{C}_{j}| \bigg) + \frac{2^{r}}{\epsilon^{r}} \sum_{j=q}^{\infty} \mathrm{E} |\widetilde{C}_{j}| \cdot \sup_{k \geq q} \mathrm{E} \bigg( \max_{1 \leq l \leq n} \bigg| \sum_{i=1}^{l}Z_{n,i-k}^{\leq} \bigg|^{r} \bigg),
\end{equation*}
with $r$ as in (\ref{eq:infmaTK}). Since $ \sum_{j=q}^{\infty}|\widetilde{C}_{j}| \leq 2 \sum_{j=q}^{\infty}|C_{j}|$, condition (\ref{eq:infmaTK}) implies that there exists a positive constant $D_{1}$ such that for all $q \in \mathbb{N}$ it holds that
\begin{equation}\label{eq:I1}
 \limsup_{n \to \infty} I_{1} \leq D_{1} \sum_{j=q}^{\infty}\mathrm{E} |C_{j}|.
\end{equation}
In order to estimate $I_{2}$ we consider separately the cases $\alpha \in (1,2)$ and $\alpha=1$. Assume first $\alpha \in (1,2)$. Applying Markov's inequality, the fact that the sequence $(C_{i})_{i \geq 0}$ is independent of $(Z_{i})$ and the stationarity of the sequence $(Z_{i})$ we obtain
\begin{eqnarray}\label{eq:alpha1}
\nonumber I_{2} & \leq &   \Pr \bigg(  \sum_{i=1}^{n} \bigg| \sum_{j=q}^{\infty} \widetilde{C}_{j}Z_{n,i-j}^{>} \bigg| > \frac{\epsilon}{2} \bigg) \leq  \frac{2}{\epsilon}\,\mathrm{E} \bigg(  \sum_{i=1}^{n} \bigg| \sum_{j=q}^{\infty} \widetilde{C}_{j}Z_{n,i-j}^{>} \bigg|  \bigg)\\[0.4em]
 & \leq & \frac{2n}{\epsilon a_{n}}  \sum_{j=q}^{\infty} \mathrm{E} |\widetilde{C}_{j}| \cdot \mathrm{E} \Big( |Z_{1}| 1 _{\{ |Z_{1}|>a_{n} \}} \Big)
\end{eqnarray}
Since by Karamata's theorem
$$ \lim_{n \to \infty} \frac{n}{ a_{n}} \mathrm{E} \Big( |Z_{1}| 1 _{\{ |Z_{1}|>a_{n} \}} \Big) = \frac{\alpha}{\alpha-1},$$
 from (\ref{eq:alpha1}) we conclude that there exists a positive constant $D_{2}$ such that
\begin{equation}\label{eq:I2a}
 \limsup_{n \to \infty} I_{2} \leq D_{2} \sum_{j=q}^{\infty} \mathrm{E} |C_{j}|.
\end{equation}
Now assume $\alpha =1$. Markov's inequality implies
$$ I_{2} \leq  \frac{2^{\delta}}{\epsilon^{\delta}}  \mathrm{E} \bigg(  \sum_{i=1}^{n} \bigg| \sum_{j=q}^{\infty} \widetilde{C}_{j}Z_{n,i-j}^{>} \bigg| \bigg)^{\delta},$$
with $\delta$ as in relation (\ref{e:momcond}). Since $\delta < 1$, a double application of the triangle inequality $|\sum_{i=1}^{\infty}a_{i}|^{s} \leq \sum_{i=1}^{\infty}|a_{i}|^{s}$ with $s \in (0,1]$ yields
\begin{eqnarray*}
I_{2} & \leq & \frac{2^{\delta}}{\epsilon^{\delta}} \sum_{i=1}^{n} \mathrm{E} \bigg( \bigg| \sum_{j=q}^{\infty} \widetilde{C}_{j}Z_{n,i-j}^{>} \bigg|^{\delta} \bigg)
       \leq  \frac{2^{\delta}}{\epsilon^{\delta} a_{n}^{\delta}} \sum_{i=1}^{n} \sum_{j=q}^{\infty} \mathrm{E}  \bigg( \bigg| \widetilde{C}_{j}Z_{i-j} 1_{\{ |Z_{i-j}|>a_{n} \}} \bigg|^{\delta} \bigg).
\end{eqnarray*}
Using again the fact that $(C_{i})$ is independent of $(Z_{i})$ and the stationarity of $(Z_{i})$ we obtain
$$ I_{2} \leq \frac{2^{\delta} n}{\epsilon^{\delta} a_{n}^{\delta}} \mathrm{E} \Big( |Z_{1}|^{\delta} 1 _{\{ |Z_{1}|>a_{n} \}} \Big) \sum_{j=q}^{\infty} \mathrm{E} |\widetilde{C}_{j}|^{\delta}.$$
From this, since by Karamata's theorem
$$ \lim_{n \to \infty} \frac{n}{a_{n}^{\delta}} \mathrm{E} \Big( |Z_{1}|^{\delta} 1 _{\{ |Z_{1}|>a_{n} \}} \Big) = \frac{1}{1-\delta},$$
it follows that there exists a positive constant $D_{3}$ such that
\begin{equation*}\label{eq:I2b}
\limsup_{n \to \infty} I_{2} \leq D_{3} \sum_{j=q}^{\infty} \mathrm{E} |C_{j}|^{\delta}.
\end{equation*}
This together with (\ref{eq:I1I2}), (\ref{eq:I1}) and (\ref{eq:I2a}) shows that
$$ \limsup_{n \to \infty}\Pr[d_{M_{2}}(V_{n}, V_{n,q})> \epsilon] \leq D_{1} \sum_{j=q}^{\infty} \mathrm{E}|C_{j}| + (D_{2}+D_{3}) \sum_{j=q}^{\infty} \mathrm{E}|C_{j}|^{s},$$
 where
 $$ s = \left\{ \begin{array}{cc}
                                   \delta, & \quad \textrm{if} \ \alpha = 1,\\[0.4em]
                                   1, & \quad \textrm{if} \ \alpha \in (1,2).
                                 \end{array}\right.$$
Finally, the dominated convergence theorem and conditions (\ref{e:momcond}) and (\ref{eq:infmaTK3}) imply relation (\ref{e:SlutskyINFV}) for $\alpha \in [1,2)$. Therefore we conclude that $L_{n}(\,\cdot\,) \dto L(\,\cdot\,)$ in $D^{1} \times D^{1}_{\uparrow}$ with the weak $M_{2}$ topology.
\end{proof}

\begin{rem}
Condition (\ref{eq:infmaTK}) holds with $r=2$ when the sequence $(Z_{i})$ is an i.i.d.~or $\rho$--mixing sequence with $\sum_{i=1}^{\infty}\rho(2^{i}) < \infty$, where
$$ \rho(n) = \sup \{ | \textrm{corr}(f,g) | : f \in \mathrm{L}^{2}(\mathcal{F}_{1}^{k}), g \in \mathrm{L}^{2}(\mathcal{F}_{k+n}^{\infty}), k=1,2,\ldots \}$$
(see Tyran-Kami\'{n}ska~\cite{Ty10b}).

In the case when the sequence of coefficients $(C_{j})$ is deterministic, conditions (\ref{e:mod1}) and (\ref{eq:infmaTK3}) can be dropped since they are implied by (\ref{e:momcond}), but this in general does not hold when the coefficients are random (see Remark 3.1 in Krizmani\'{c}~\cite{Kr22-1}).
\end{rem}

\section*{Acknowledgment}
 This work has been supported in part by University of Rijeka research grants uniri-prirod-18-9 and uniri-pr-prirod-19-16 and by Croatian Science Foundation under the project IP-2019-04-1239.


\end{document}